\documentclass[12pt]{amsart}
\usepackage{amscd,amsmath,amsthm,amssymb,graphics}
\usepackage{slashbox}
\usepackage{pgf,tikz, pifont}

\usetikzlibrary{arrows}

\usepackage[all]{xy}
\usepackage{lineno}

\def\frk{\frak}               

\def\Phi{{\frk n}}
\def\Phi{{\frk N}}
\def\opn#1#2{\def#1{\operatorname{#2}}} 
\opn\chara{char} \opn\length{\ell} \opn\pd{pd} \opn\rk{rk}
\opn\projdim{proj\,dim} \opn\injdim{inj\,dim} \opn\rank{rank}
\opn\depth{depth} \opn\grade{grade} \opn\height{height}
\opn\embdim{emb\,dim} \opn\codim{codim}

\opn\Tr{Tr} \opn\bigrank{big\,rank}
\opn\superheight{superheight}\opn\lcm{lcm}
\opn\trdeg{tr\,deg}
\opn\reg{reg} \opn\lreg{lreg} \opn\ini{in} \opn\lpd{lpd}
\opn\size{size}\opn\bigsize{bigsize}
\opn\cosize{cosize}\opn\bigcosize{bigcosize}
\opn\sdepth{sdepth}\opn\sreg{sreg}
\opn\link{link}\opn\fdepth{fdepth}
\opn\index{index}
\opn\index{index}
\opn\indeg{indeg}
\opn\N{N}
\opn\SSC{SSC}
\opn\SC{SC}
\opn\lk{lk}
\opn\div{div} \opn\Div{Div} \opn\cl{cl} \opn\Cl{Cl}
\opn\Spec{Spec} \opn\Supp{Supp} \opn\supp{supp} \opn\Sing{Sing}
\opn\Ass{Ass} \opn\Min{Min}\opn\Mon{Mon} \opn\dstab{dstab} \opn\astab{astab}
\opn\Syz{Syz}
\opn\reg{reg}
\opn\Ann{Ann} \opn\Rad{Rad} \opn\Soc{Soc}
\opn\Im{Im} \opn\Ker{Ker} \opn\Coker{Coker} \opn\Am{Am}
\opn\Hom{Hom} \opn\Tor{Tor} \opn\Ext{Ext} \opn\End{End}\opn\Der{Der}
\opn\Aut{Aut} \opn\id{id}

\opn\nat{nat}
\opn\pff{pf}
\opn\Pf{Pf} \opn\GL{GL} \opn\SL{SL} \opn\mod{mod} \opn\ord{ord}
\opn\Gin{Gin} \opn\Hilb{Hilb}\opn\sort{sort}
\opn\initial{init}
\opn\ende{end}
\opn\height{height}
\opn\type{type}
\opn\aff{aff} \opn\con{conv} \opn\relint{relint} \opn\st{st}
\opn\lk{lk} \opn\cn{cn} \opn\core{core} \opn\vol{vol}
\opn\link{link} \opn\Link{Link}\opn\lex{lex}
\opn\gr{gr}

\def\pot#1#2{#1[\kern-0.28ex[#2]\kern-0.28ex]}

\opn\dirlim{\underrightarrow{\lim}}
\opn\inivlim{\underleftarrow{\lim}}
\def\Implies{\ifmmode\Longrightarrow \else
        \unskip${}\Longrightarrow{}$\ignorespaces\fi}
\def\implies{\ifmmode\Rightarrow \else
        \unskip${}\Rightarrow{}$\ignorespaces\fi}
\def\iff{\ifmmode\Longleftrightarrow \else
        \unskip${}\Longleftrightarrow{}$\ignorespaces\fi}

\let\:=\colon
\newtheorem{Theorem}{Theorem}[section]
 \newtheorem{Lemma}[Theorem]{Lemma}
 \newtheorem{Corollary}[Theorem]{Corollary}
 \newtheorem{Proposition}[Theorem]{Proposition}

 \newtheorem{Example}[Theorem]{Example}
 
 \newtheorem{Definition}[Theorem]{Definition}

  \newtheorem{Construction}[Theorem]{Construction}
 
\newtheorem{Notation}[Theorem]{Notation}
\newtheorem{Conventions}[Theorem]{Conventions}
\let\epsilon\varepsilon
\let\kappa=\varkappa
\def\qed{\ifhmode\textqed\fi
      \ifmmode\ifinner\quad\qedsymbol\else\dispqed\fi\fi}
\def\textqed{\unskip\nobreak\penalty50
       \hskip2em\hbox{}\nobreak\hfil\qedsymbol
       \parfillskip=0pt \finalhyphendemerits=0}
\def\dispqed{\rlap{\qquad\qedsymbol}}

\opn\dis{dis}
\def\pnt{{\raise0.5mm\hbox{\large\bf.}}}

\opn\Lex{Lex}

\begin{document}

 \title{The  resolutions of generalized co-letterplace ideals and their powers}
\author{Dancheng Lu}

\address{School  of Mathematical Sciences, Soochow University, 215006 Suzhou, P.R.China}
\email{ludancheng@suda.edu.cn}

\author{Zexin Wang}

\address{School  of Mathematical Sciences, Soochow University, 215006 Suzhou, P.R.China}
\email{zexinwang6@outlook.com}

 \begin{abstract} We present a natural and explicit multigraded minimal free resolution for each generalized co-letterplace ideal (see Definition~\ref{g}). Our resolution differs significantly from the ones presented in the works of Ene et al. \cite{EHM} and D'Al{`\i} et al. \cite{DFN}. Additionally, we show that each power of a large class of generalized co-letterplace ideals can be represented as the quotient of another generalized co-letterplace ideal by a regular sequence of variable differences.  Finally, we provide a new class of simplicial spheres.
 \end{abstract}


 \maketitle

\section{Introduction}

In the field of combinatorial commutative algebra, a significant research focus is on the search for the minimal multigraded resolutions for monomial ideals in a polynomial ring over a field under various restrictions. On the other hand, it is worth noting that the concepts of letterplace and co-letterplace ideals of partially ordered sets (posets) have been extensively explored and investigated in the literatures, as seen in \cite{DFN, DFN2, EHM, FGH} and \cite{JKM}. These ideals have been found to be closely related to numerous monomial ideals that have been studied in the other literatures, through the process of quotienting these ideals with a regular sequence of variable differences.  Notably, every co-letterplace ideal possesses linear quotients and, therefore, has a linear resolution. This characteristic has sparked significant interest in obtaining explicit minimal resolutions for such ideals.

 Let $P$ be a finite poset, and $n$ be a positive integer.  Let $R$  be the polynomial ring  $\mathbf{K}[X_{p,i}\mid \ p\in P,  i\in [n] ]$ over a field $\mathbf{K}$ in $|P|\times n$ variables.  Here,  $[n]$ represents the set $\{1,2,\ldots,n\}$. A map $f:P\rightarrow [n]$ is said to be {\it isotone} if $f(p)\leq f(q)$ whenever $p\leq q$ for all $p,q\in P$.  We denote $\mathrm{Hom}(P,n)$ as the set of all isotone maps  from $P$ to $[n]$.
To each map $f\in \mathrm{Hom}(P,n)$, we associate a monomial $$U_f:=\prod_{p\in P} X_{p, f(p)}\in R.$$
Then, the  ideal $\mathcal{L}(P,n)$ is defined as follows: $$\mathcal{L}(P,n):=(U_f\mid f\in \mathrm{Hom}(P,n))\subseteq R.$$
A {\it poset ideal} of a poset $P$ is a non-empty subset $I$ of $P$ such that if $y\in I$ and $x\in P$ with $x\leq y$ then $x\in I$.  Let $\mathfrak{A}$ be a poset ideal of  $\mathrm{Hom}(P,n)$, which is  a poset under the natural relation  $f\leq g$ if and only if $f(p)\leq g(p)$ for all $p\in P$.
The ideal $\mathcal{L}(P,n; \mathfrak{A})$ is given as: $$\mathcal{L}(P,n; \mathfrak{A}):=(U_f\mid f\in \mathfrak{A})\subseteq R.$$
Both  $\mathcal{L}(P,n)$ and $\mathcal{L}(P,n; \mathfrak{A})$ are referred as {\it co-letterplace} ideals.

In \cite[Theorem 3.6]{EHM}, an explicit multigraded minimal free resolution for the ideal $\mathcal{L}(P,n)$ is provided. This result is based on the observation that $\mathcal{L}(P,n)$ always admits a regular decomposition function, as defined in \cite{HT02}.

In \cite{Yan2000}, it was introduced a method for constructing minimal free resolutions of squarefree modules with a linear resolution. This method assigns a complex of finitely generated free $R$-modules, denoted by $\mathbb{F}.(M)$, to each finitely generated squarefree module $M$. If $M$ is a Cohen-Macaulay squarefree module, then the canonical  module of $M$, denoted by $\omega_M$, is also squarefree. Moreover, the complex $\mathbb{F}.(\omega_M)$ provides a minimal free resolution of the Alexander dual module $M^*$ of $M$.

 Building on this approach, an explicit multigraded minimal free resolution of $\mathcal{L}(P,n; \mathfrak{A})$ was presented in \cite[Theorem 4.5]{DFN}. This resolution is expressed in terms of the module structure of  the Cohen-Macaulay module $\omega_{R/\mathcal{L}^{\vee}(P,n; \mathfrak{A})}$, where $\mathcal{L}^{\vee}(P,n; \mathfrak{A})$ denotes the Alexander dual of $\mathcal{L}(P,n; \mathfrak{A})$.

  We now extend the notion of a co-letterplace ideal. Let $P$, $n$ and $R$ be the same as mentioned above.
  Consider a map $\mathcal{A}$ from  $P$ to $2^{[n]}\setminus \{\emptyset\}$, which assigns each $p\in P$ with a non-empty subset $\mathcal{A}_p$ of $[n]$. We define $\mathrm{Hom}(P,\mathcal{A})$ to be the set of all the isotone maps from $P$ to $[n]$ such that $f(p)\in \mathcal{A}_p$ for all $p\in P$.    Then, $\mathrm{Hom}(P, \mathcal{A})$ is also a poset naturally. Given a poset ideal  $\mathfrak{A}$  of $\mathrm{Hom}(P, \mathcal{A})$, let $\mathcal{L}(P, \mathcal{A};\mathfrak{A})$ (or $\mathcal{L}(\mathfrak{A})$ for short) denote the  monomial ideal in $R$ generated by $U_f$ with $f\in \mathfrak{A}$. Namely, $$\mathcal{L}(P, \mathcal{A};\mathfrak{A}):=(U_f\mid f\in \mathfrak{A})\subseteq R.$$
If $\mathfrak{A}=\mathrm{Hom}(P, \mathcal{A})$, then $\mathcal{L}(P, \mathcal{A};\mathfrak{A})$ is denoted by $\mathcal{L}(P, \mathcal{A})$ for short.

\begin{Definition} \label{g} \em We call a monomial ideal of the  form $\mathcal{L}(P, \mathcal{A};\mathfrak{A})$  a {\it generalized co-letterplace ideal}. The {\it size} of a  generalized co-letterplace ideal $\mathcal{L}=\mathcal{L}(P, \mathcal{A};\mathfrak{A})$ is defined as the sum of sizes of the sets $\mathcal{A}_p$ for all $p\in P$. Namely, $$\mathrm{size}(\mathcal{L}):=\sum_{p\in P}|\mathcal{A}_p|.$$
\end{Definition}

In this context, $\mathcal{A}_p$ is merely required to be a non-empty subset of $[n]$ for all $p \in P$, without the additional stipulation that it must form an interval. This may seem somewhat counterintuitive, given the significance of the total order on $[n]$.

This generalization has two advantages. First, within the category of   generalized co-letterplace ideal, every ideal $L$   has a Betti-splitting $L=J+K$ such that $J$, $K$ and $J\cap K$ all belong to this category, but with smaller size. This allows for a more efficient study of these ideals.  Secondly, the simplicial complexes associated with such ideals are also simplicial balls or spheres, as is the case with co-letterplace ideals in \cite{DFN}. This provides a larger class of simplicial spheres and expands the scope of possible applications.

In order to simplify the exposition and unify the notation, we introduce the following conventions.
\begin{Conventions}\label{setup} Throughout this paper, we always assume that $P$ is a finite poset and  $\mathcal{A}$ is a map from $P$ to the set of non-empty subsets of $[n]$.  Let $\mathfrak{A}$ be a poset ideal of $\mathrm{Hom}(P,\mathcal{A})$.

We will denote $P$ as $\{p_1,\ldots,p_m\}$, where if $p_i<p_j$ then $i<j$; and write $A_i$ for $\mathcal{A}_{p_i}$ for all possible $i,j$. Accordingly, sometimes we  write $\mathrm{Hom}(P, A_1,\ldots,A_m)$ for $\mathrm{Hom}(P,\mathcal{A})$ and $\mathrm{Hom}(P, A_1,\ldots,A_m; \mathfrak{A})$ for $\mathrm{Hom}(P,\mathcal{A}; \mathfrak{A})$.

We also employ the vector $(a_1,\ldots,a_m)$ with $a_i\in [n]$ to represent the map $f:P\rightarrow [n]$ such that $f(p_i)=a_i$ for all $i=1,\ldots,m$.

For a non-empty subset $K$ of $[n]$, we use  $\widehat{K}$ to represent $\max(K)$, the maximum of $K$, and use $X_{p_i,K}$ to denote the monomial $\prod_{a\in K}X_{p_i,a}$.
\end{Conventions}

The main result of this paper is as follows:
\begin{Theorem} \label{main} {\em (Theorem~\ref{main4.4})} Under  Conventions~\ref{setup}, the ideal   $\mathcal{L}(\mathfrak{A})$ has a multigraded  minimal free resolution of the following form:
\begin{equation*}\mathbb{F}: \xymatrix{0\ar[r] &F_p\ar[r]^{d_p}&F_{p-1}\ar[r]^{d_{p-1}}&\cdots\ar[r]^{d_1}&F_0\ar[r]^{d_0}&\mathcal{L}(\mathfrak{A})\ar[r]&0.}
\end{equation*} Here, the bases of $F_t, t=0,\ldots,p$ are given as follows:
\begin{itemize}
\item For $0\leq t\leq p$, the basis of $F_t$ consists of  the symbols  $[K_1,\ldots, K_m]$ with the multidegree $X_{p_1,K_1}\cdots X_{p_m,K_m}$, satisfying  following conditions: \begin{enumerate}
                                                              \item $\emptyset\neq K_1\subseteq A_1,\ldots,  \emptyset\neq  K_m\subseteq A_m$;
                                                               \item $|K_1|+\cdots+|K_m|=m+t$;
                                                              \item $\max(K_i)\leq \min(K_j)$ if $p_i<p_j$ for all $1\leq i<j\leq m$;
                                                              \item $(\widehat{K_1}, \widehat{K_2}, \ldots, \widehat{K_m})\in \mathfrak{A}$.
                                                            \end{enumerate}
                                                            It should be noted that the integer $p$, i.e., the projective dimension of $\mathcal{L}(\mathfrak{A})$, is the maximum value of $t$ for which there exist $K_1,\ldots,K_m$ satisfying the four conditions mentioned above.

The differential maps $d_t, t=0,\ldots,p $ are as follows:

    \item If $[K_1,\cdots,K_m]\in F_0$, then $d_0([K_1,\cdots,K_m])=X_{p_1,K_1}X_{p_2,K_2}\ldots X_{p_m,K_m}\in \mathcal{L}(\mathfrak{A});$
    \item If $t\geq 1$ and $[K_1,\cdots,K_m]\in F_t$, we denote $$\Omega=\{X_{p_i,a}\mid 1\leq i\leq m, |K_i|\geq 2, a\in K_i\}$$ and  $$\overline{\Omega}=\{X_{p_i,a}\mid 1\leq i\leq m,  a\in K_i\}.$$
 Define a total order on $\overline{\Omega}$: $X_{p_i,a}>X_{p_j,b}$ if and only if either $i<j$ or $i=j$ and $a<b$. Let $\sigma(X_{p_i,a},\overline{\Omega})$ be the number of elements in $\overline{\Omega}$ that are  larger than  $X_{p_i,a}$. Then  $$d_t([K_1,\ldots,K_m])=\sum_{X_{p_i,a}\in \Omega}(-1)^{\sigma(X_{p_i,a},\overline{\Omega})}X_{p_i,a}[K_1, \ldots, K_i\setminus\{a\},\ldots, K_m].$$
    \end{itemize}
\end{Theorem}

By comparing the resolutions presented in \cite[Theorem 3.6]{EHM} and \cite[Theorem 4.5]{DFN}, it becomes evident that the resolution provided in Theorem~\ref{main} is more accessible and natural.

The main results of this paper are summarized as follows. In addition to Theorem~\ref{main}, we demonstrate in Corollary~\ref{powersc} that the multigraded minimal free resolution of all powers of a generalized co-letterplace ideal of the form $\mathcal{L}(P,\mathcal{A})$  can also be explicitly expressed. Furthermore, we establish that the simplicial complex associated with a generalized co-letterplace ideal is either a simplicial ball or a simplicial sphere, thereby providing a large class of simplicial spheres. (It should be noted that the boundary of a simplicial ball is a simplicial sphere.)

The paper is organized as follows. Section 2 provides an overview of the basic notions and results. In Section 3, we prove that every generalized co-letterplace ideal has linear quotients and determine its homological shift ideals, which are shown to not have a linear resolution in general. Section 4 presents our main result, which is established through an induction process. In Section 5, we demonstrate that every power of a generalized co-letterplace ideal of the form $\mathcal{L}(P,\mathcal{A})$ is isomorphic to the quotient of another generalized co-letterplace ideal modulo a regular sequence consisting of the differences between two variables. However, the polarization of a power of a generalized co-letterplace ideal is not itself a generalized co-letterplace ideal in general. Finally, in  Section 6, we show that the simplicial complex associated with a generalized co-letterplace ideal is either a simplicial ball or a simplicial sphere, and we provide a classification for when it is a simplicial sphere in terms of the projective dimension of  the corresponding generalized co-letterplace ideal.

\section{Preliminaries}

In this section  we collect  basic notions and results we need in this article.
\subsection{Mapping cone}\label{mapping}
    The mapping cone is a method used to construct complexes, particularly useful for building long exact sequences in homological algebra. Given a chain map $f: \mathbb{A}\rightarrow \mathbb{B}$,  where  $\mathbb{A}=(A_n, d_n)_{n\in \mathbf{Z}}$ and  $\mathbb{B}=(B_n, \delta_n)_{n\in \mathbf{Z}}$ are chain complexes, we can construct a new complex by forming the mapping cone of this map.

Specifically, the mapping cone  $\mathbb{C}(f)=(C_n,\partial_n)$ is constructed as follows:

 (1)  $C_n=A_{n-1}\oplus B_{n}$ for $n\in \mathbf{Z}$;

 (2) $\partial_n=\begin{pmatrix}-d_{n-1}&0\\ f_{n-1}& \delta_n\end{pmatrix}$, i.e., $\partial_n \begin{pmatrix} a\\b \end{pmatrix}=\begin{pmatrix} - d_{n-1}(a)\\f_{n-1}(a)+\delta_n(b) \end{pmatrix}$ for any $a\in A_{n-1}$ and $b\in B_n$. In other words, $$\partial_n|_{A_{n-1}}=-d_{n-1}+f_{n-1}: A_{n-1}\rightarrow A_{n-2}\oplus B_{n-1} \mbox{ and }\partial_n|_{B_{n}}=\delta_n: B_{n}\rightarrow B_{n-1}.$$ This construction gives rise to a short exact sequence of complexes: $$0\rightarrow \mathbb{B}\rightarrow \mathbb{C}(f)\rightarrow \mathbb{A}[-1]\rightarrow 0.$$
  Given a short exact sequence of $R$-modules $0\rightarrow A\stackrel{f_{-1}}{\rightarrow} B\rightarrow C\rightarrow 0$ and suppose $\mathbb{A}$ and $\mathbb{B}$ are free resolutions of $A$ and $B$ respectively. There exists a chain map $f:\mathbb{A}\rightarrow \mathbb{B}$ that lifts $f_{-1}$. One can check that $\mathbb{C}(f)$ is a free resolution of $C$.  We refer to \cite[Page 106]{P} for more details on mapping cone.

 \subsection{Betti splitting}\label{betti}

 Let $I, J$ and $K$ be monomial ideals such that $G(I)$ is the disjoint union of $G(J)$ and $G(K)$. From the short exact sequence $$0\rightarrow J\cap K\stackrel{f_{-1}}{\rightarrow} J\oplus K\rightarrow I\rightarrow 0,$$
 together with the mapping cone construction applied to the map $f_{-1}$, we derive the following inequality: $$\beta_{i,j}(I)\leq \beta_{i,j}(J)+\beta_{i,j}(K)+\beta_{i-1,j}(J\cap K)   \mbox{ for all } i, j\geq 0.$$

 \begin{Definition} \em (see \cite{FHT})  Let $I,J$ and $K$ as above.  We call that $I=J+K$ is a {\it Betti-splitting} if $$\beta_{i,j}(I)= \beta_{i,j}(J)+\beta_{i,j}(K)+\beta_{i-1,j}(J\cap K)  \mbox{ for all } i, j\geq 0.$$
 \end{Definition}

 It is evident that $I=J+K$ is a Betti-splitting if and only if $\mathbb{C}(f)$ is a graded minimal free resolution of $I$. Here, $f$ is a chain map that lifts $f_{-1}$ to  the graded minimal free resolutions of  $J\cap K$  and  $J\oplus K$ .

 A sufficient condition for $I=J+K$ to be a Betti-splitting is as follows, see \cite[Proposition 3.1]{B}.

 \begin{Lemma} \label{component} Assume that $I,J$ and $K$ are as mentioned above.  If both $J$ and $K$ have linear resolutions, then $I=J+K$ forms a Betti-splitting.
 \end{Lemma}

 \subsection{Homological shift ideals}
We quote the next definition from \cite{CF}.
\begin{Definition}\rm Let $I\subseteq R$ be a monomial ideal with minimal multigraded free resolution
$$
\mathbb{F}\ \ :\ \ 0 \rightarrow F_p \rightarrow F_{p-1} \rightarrow \cdots\rightarrow F_0 \rightarrow I\rightarrow 0,
$$
where $F_i = \bigoplus_{j=1}^{\beta_i(I)}S(-\mathbf{a}_{i,j})$. The vectors $\mathbf{a}_{i,j}\in\mathbf{Z}^n$, $i\ge0$, $j=1,\dots,\beta_i(I)$, are called the $i$th \textit{multigraded shifts} of $I$. The monomial ideal
	$$
	\mbox{HS}_i(I)\ :=\ ({X}^{\mathbf{a}_{i,j}}\mid j=1,\dots,\beta_i(I))
	$$
	is called the $i$th \textit{homological shift ideal} of $I$.
\end{Definition}

Note that $\mbox{HS}_0(I)=I$ and $\mbox{HS}_i(I)=(0)$ for $i<0$ or $i>\pd(I)$.\medskip

The main purpose of the theory of homological shift ideals is to understand what homological and combinatorial properties are enjoyed by all $\mbox{HS}_t(I)$, $t=0, \ldots, \pd(I)$.
\subsection{Homology spheres and balls} Due to its connection to the Upper Bound Conjecture, the concept of a simplicial sphere has been extensively studied in the literatures (see  \cite{BH} and also \cite{Bier,BjZi}). Recall that a simplicial complex $\Delta$ is a {\it simplicial sphere} (resp. {\it simplicial ball})  if $|\Delta|$, the geometric realization of $\Delta$, is homeomorphic to $S^d$ (resp. $B^d$) for some $d$. Here, $S^d=\{v\in \mathbf{R}^{d+1}\mid |v|=1\}$ and $B^d=\{v\in \mathbf{R}^d\mid |v|\leq 1\}$.  In \cite[page 1859]{Bj}, the explicit definitions of piecewise linear (PL for short) polyhedral balls and  spheres are provided. Restricting to simplicial complexes, we see that a simplicial complex $\Delta$ is a PL-ball precisely when a subdivision of $\Delta$  is a subdivision of a simplex. Therefore,  for a simplicial complex $\Delta$, it has to be a simplicial ball (resp. sphere)  if it is a PL ball (resp. sphere). In \cite{DFN}, the following notions are introduced as  generalizations of simplicial balls and spheres.
\begin{Definition}\em
Let $\Delta$ be a simplicial complex of dimension $d$ and let $\mathbf{K}$ be a field. We say $\Delta$ is a \emph{homology sphere} over $\mathbf{K}$ if, for every face $F \in \Delta$, one has that $\lk_{\Delta}F$ has the (reduced) homology over $\mathbf{K}$ of a $\mathrm{dim}(\mbox{lk}_{\Delta}F)$-dimensional sphere. In other words,
\[\widetilde{H}_i(\mbox{lk}_{\Delta}F, \mathbf{K}) \cong \begin{cases} \mathbf{K} & i = \mathrm{dim}(\mbox{lk}_{\Delta}F) \\
                               0 & i \neq \mathrm{dim}(\mbox{lk}_{\Delta}F)
                   \end{cases}
\]
for every face $F$ of $\Delta$.

We say $\Delta$ is a \emph{homology ball} over $\mathbf{K}$ if there exists a subcomplex $\Sigma$ of $\Delta$ such that:
\begin{itemize}
\item $\Sigma$ is a $(d-1)$-dimensional homology sphere over $\mathbf{K}$;
\item $\mbox{lk}_{\Delta}F$ is a homology  sphere if $F \notin \Sigma$ and otherwise it has zero homology.
\end{itemize}
Note that the facets of $\Sigma$ are those codimension $1$ faces of $\Delta$
that are contained in exactly one facet of  $\Delta$. We therefore refer to $\Sigma$ as
the $\emph{boundary}$ of $\Delta$, denoted by $\partial\Delta = \Sigma$.

\end{Definition}
\section{Linearity}
In this section, we  establish that every generalized co-letterplace ideal is weakly polymatoidal and thus has a linear resolution. Additionally, we present the minimal generating sets of the homological shift ideals associated with co-letterplace ideals. Furthermore, it is shown that these homological shift ideals do not have a linear resolution in general.

In the following definition, let $S:=\mathbf{K}[X_1,\ldots, X_t]$. If  $I$ is a monomial ideal, then $G(I)$ denotes the set of minimal monomial generators of $I$.

\begin{Definition} \label{w} \em Following \cite{KH}, we say that a monomial ideal $I$ of $S$  generating in a single degree   is {\it weakly polymatroidal} if for every pair of elements $u=X_1^{a_1}\cdots X_t^{a_t}$ and $v=X_1^{b_1}\cdots X_t^{b_t}$ of $G(I)$ with $a_1=b_1,\ldots,a_{q-1}=b_{q-1}$ and $a_q<b_q$,  there exists $q<p\leq t$
such that $w:=(X_qu)/X_p$ belongs to $G(I)$.
 \end{Definition}
 It was shown in \cite{KH} that every weakly polymatroidal monomial ideal has linear quotients and thus has a linear resolution.

\begin{Lemma} \label{weak} Let $P, n, \mathcal{A}$ and $\mathfrak{A}$ be the same as in Convention~\ref{setup}. Then $\mathcal{L}(\mathfrak{A})$ is weakly polymatroidal.
\end{Lemma}

\begin{proof}  Denote  $\mathcal{L}(\mathfrak{A})$ as $\mathcal{L}$ for short. We order the variables as follows:
$$X_{p_1,1}>X_{p_1,2}>\cdots>X_{p_1,n}>X_{p_2,1}>\cdots>X_{p_2,n}>\cdots>X_{p_m,1}>\cdots>X_{p_m,n}.$$
This is to say $X_{p_1,1}$ is the $X_1$ in Definition~\ref{w},  $X_{p_1,2}$ is the $X_2$ in Definition~\ref{w}, and so on. This order will be maintained  throughout this paper.

Let $U_f,U_g\in G(\mathcal{L})$, where  $f,g$ are distinct maps in $\mathfrak{A}$. Let $X$ be the first (maximal) variable such that $\deg_X(U_f)\neq \deg_X(U_g)$. Assume $f(p_i)=g(p_i)$ for $i=1,\ldots,k-1$ and $f(p_k)<g(p_k)$. Then $X=X_{p_k,f(p_k)}$, as $1=\deg_X(U_f)> \deg_X(U_g)=0$.  To prove $\mathcal{L}$ is weakly polymatroidal, it suffices to show that $\frac{U_gX_{p_k,f(p_k)}}{Y}\in \mathcal{L}$ for some variable $Y<X_{p_k,f(p_k)}$. To this end, define a map: $h:P\rightarrow [n]$ as follows: $$h(p_i)=\left\{
                       \begin{array}{ll}
                         g(p_i), & \hbox{$i\neq k$;} \\
                         f(p_i), & \hbox{$i=k$.}
                       \end{array}
                     \right.
$$
Let us verify that $h$ is isotone. Suppose that $p_i<p_j$. If $k\notin\{i,j\}$, then $h(p_i)=g(p_i)\leq g(p_j)=h(p_j)$; if $k=i$, then $h(p_i)=f(p_k)<g(p_k)\leq g(p_j)=h(p_j)$; if $k=j$, then $i<k$, and thus $h(p_i)=f(p_i)=g(p_i)\leq g(p_j)=h(p_j)$. Therefore, $h$ is isotone  and so $h\in \mathrm{Hom}(P, A_1,\ldots,A_m)$. Consequently, since $h\leq g$, we conclude that $h\in \mathfrak{A}$.  Note that $X_{p_k,g(p_k)}<X_{p_k,f(p_k)}$ and $U_h=\frac{U_gX_{p_k,f(p_k)}}{X_{p_k,g(p_k)}}\in \mathcal{L}$, the result follows.\end{proof}

We next present minimal generating sets of homological shift ideals of $\mathcal{L}(\mathfrak{A})$. To do this, we define a linear order $\leq_{\ell}$ as follows: for any $f,g\in \mathrm{Hom}(P, A_1,\ldots,A_m)$,
$$f\leq_{\ell} g\Longleftrightarrow \mbox{the left-most nonzero entry of } (f-g) \mbox{ is negative}. $$
Here, $(f-g)$ denotes the $m$-tuple $$(f(p_1)-g(p_1), f(p_2)-g(p_2), \cdots, f(p_m)-g(p_m))\in \mathbf{Z}^m.$$
 This order is   a refinement of  the natural  partial order $\leq$ on  $\mathrm{Hom}(P, A_1,\ldots,A_m)$. In addition, it is easy to see that  $U_f\geq_{\mathrm{lex}} U_g$ if and only if $f\leq_{\ell} g$. Assuming that $f_1<_{\ell}f_2<_{\ell}f_3<_{\ell}\cdots <_{\ell} f_t$ are all maps in $\mathfrak{A}$, then,  in view of the proof of \cite[Theorem 1.4]{KH},  we see that  $\mathcal{L}(\mathfrak{A})$ has linear quotients with respect to the sequence $U_{f_1},U_{f_2},\ldots, U_{f_t}$.  Following \cite{HT02},   we use $\mathrm{set}(U_{f_i})$ to denote the set of the variables that appear in the colon ideal $(U_{f_1}, \ldots, U_{f_{i-1}}):U_{f_i}$ for $i=2,\ldots,t$.

 \begin{Lemma} \label{set} For each $2\leq i\leq t$, the set $\mathrm{set}(U_{f_i})$ is identical  to $\Delta_{f_i}$. Here, for $f\in \mathfrak{A}$, the set
 $\Delta_f$ is defined as follows: $$\Delta_f:=\{X_{p_s,k}\mid 1\leq s\leq m, k< f(p_s), k\in A_s,  f(p_i)\leq k \mbox { if } p_i< p_s \}.$$
 \end{Lemma}

 \begin{proof} Let  $X_{p_s,k}\in \Delta_{f_i}$. We define  a map $h$ from $P$  to $[n]$  by   $$h(p_j)=\left\{
                            \begin{array}{ll}
                              f_i(p_j), & \hbox{$j\neq s$;} \\
                              k, & \hbox{$j=s$.}
                            \end{array}
                          \right.$$  Then $h$ is isotone.
Since $h(p)\leq f_i(p)$  for all $p\in P$, it follows that $h$ belongs to $\mathfrak{A}$ and so $h=f_j$ for some $j< i$. Note that $U_h:U_{f_i}=(X_{p_s,k})$, we obtain that $X_{p_s,k}\in \mathrm{set}(U_{f_i})$, implying $\Delta_{f_i}\subseteq \mathrm{set}(U_{f_i})$.
Conversely, let $X_{p_s,k}\in \mathrm{set}(U_{f_i})$. Then, there exists $1\leq j<i$ such that  $U_{f_j}:U_{f_i}$ is generated $X_{p_s,k}$. Note that $f_j(p)=f_i(p)$ for all $p\in P$ with $p\neq p_s$ and $f_j(p_s)=k$. From this, it follows that $X_{p_s,k}$ belongs to $\Delta_{f_i}$, as required.
 \end{proof}

According to \cite[Lemma 1.5]{HT02},  the symbols $(f,\sigma)$,  where  $f\in \mathfrak{A}$ and $\sigma\subseteq \Delta_f,  |\sigma|=i$ form a free basis for the $i$-th syzygy of the ideal $\mathcal{L}(\mathfrak{A})$. Furthermore,   each symbol $(f,\sigma)$ has a multidegree of  $U_fX_{\sigma}$, where $X_{\sigma}:=\prod\limits_{X_{p_i,k}\in \sigma} X_{p_i,k}$.

\begin{Theorem}\label{generator} For $t\geq 0$, the minimal generating set of $\mathrm{HS}_t(\mathcal{L(\mathfrak{A})})$ consists of all monomials $$X_{p_1,K_1}X_{p_2,K_2}\cdots X_{p_m,K_m}$$ such that
\begin{enumerate}
  \item $\emptyset \neq K_i\subseteq A_i$ for all $i\in [m]$;
  \item  $\max (K_{i})\leq \min (K_j)$ if $p_{i}<p_j$;
    \item  $|K_1|+\cdots+|K_m|=m+t$;
    \item  the map $(\widehat{K_1}, \widehat{K_2}, \ldots, \widehat{K_m})$ belongs to $\mathfrak{A}$. \end{enumerate}

\end{Theorem}

\begin{proof} Denote $\mathcal{L}(\mathfrak{A})$ as $\mathcal{L}$. Let $U_fX_{\sigma}$ be a minimal generator of $\mathrm{HS}_t(\mathcal{L})$. Here, $f\in \mathfrak{A}$, $\sigma\subseteq \Delta_f$ and $|\sigma|=t$. For $i\in [m]$, we put $$K_i=\{a\in [n]\:\; X_{p_i, a}\in \sigma\}\cup\{f(p_i)\}.$$ Then
$$U_fX_{\sigma}=X_{p_1,K_1}X_{p_2,K_2}\cdots X_{p_m,K_m}.$$
It is enough to show that $\max (K_{\ell})\leq \min (K_j)$ if $p_{\ell}<p_j$. To check this, suppose $p_{\ell}<p_j$. Note that $f(p_\ell)=\max (K_{\ell})$. Moreover, if denote $k=\min (K_j)$, then either $X_{p_j,k}\in \Delta_f$ or $k=f(p_j)$, and it follows that $f(p_{\ell})\leq k$. This proves every minimal generator of $\mathrm{HS}_t(\mathcal{L})$ has the form described in the theorem.

Conversely, let's consider a monomial  $U=X_{p_1,K_1}X_{p_2,K_2}\cdots X_{p_m,K_m}$ that satisfies the four conditions given in the theorem. We define $f:P\rightarrow [n]$ by $f(p_i)=\max(K_i)$ for all $i\in [m]$. Then $f\in \mathfrak{A}$.  Let $$\sigma=\{X_{p_i,k}\:\; 1\leq i\leq m, k\in K_i\setminus\{f(p_i)\}\}.$$
Then $\sigma\subseteq \Delta_f$ and $|\sigma|=t$. It follows that $U=U_fX_{\sigma}$ and it is a minimal generator of $\mathrm{HS}_t(\mathcal{L})$. This completes the proof.
\end{proof}

It was proved in \cite{FH} that if $I$ has linear quotients, then so does  $\mathrm{HS}_1(I)$. Furthermore, in \cite{CF}, it was shown that $\mathrm{HS}_t(\mathcal{L}(P,2))$ has linear quotients for all $t\geq 1$.  However, the following result demonstrates that these results cannot be improved.

  Let $I$ be a monomial ideal generated in a single degree and $G(I)$ the minimal set of monomial generators of $I$. Recall from \cite{L} that $I$ is {\it  quasi-linear} if for each $u\in G(I)$, the colon ideal $I\setminus_u:u$ is generated by variables. Here, $I\setminus_u$ is the ideal generated by monomials in $G(I)\setminus \{u\}.$  It was shown in \cite[Theorem 3.3]{L} that if $I$ has a linear resolution then $I$ is quasi-linear.

\begin{Proposition} Let $P=[m]$, i.e., $P$ is a chain of length $m-1$. If $m\geq 3$ and $n\geq 3$, then $\mathrm{HS}_2(\mathcal{L}(P,n))$ does not admit a linear resolution.
\end{Proposition}

\begin{proof} For monomials $\mathbf{a}$ and $\mathbf{b}$,  we define $\mathbf{b} / \mathbf{a}$ as the monomial obtained by dividing $\mathbf{b}$ by their greatest common divisor $(\mathbf{a},\mathbf{b})$. In view of \cite[Theorem 3.3]{L}, we only need  to show that $\mathrm{HS}_2(\mathcal{L}(P,n))$ is not quasi-linear.
 Put $$\alpha:=X_{1,1}X_{1,2}X_{1,3}X_{2,3}\prod_{i=3}^mX_{i,3} \mbox{\quad and \quad } \beta:=X_{1,1}X_{2,1}X_{3,1}X_{3,2}\prod_{i=3}^mX_{i,3}.$$ By Theorem~\ref{generator}, $\alpha$ and $\beta$ are minimal generators of $\mathrm{HS}_2(\mathcal{L}(P,n))$. Note that, applying our definition, $\beta / \alpha = X_{2,1}X_{3,1}X_{3,2}$.
Assume, for contradiction, that $\mathrm{HS}_2(\mathcal{L}(P,n))$ is quasi-linear. Then there should exist a minimal generator of $\mathrm{HS}_2(\mathcal{L}(P,n))$, say $\gamma$, such that $\gamma / \alpha$ is a variable belonging to $\{X_{2,1}, X_{3,1}, X_{3,2}\}$. However, it is straightforward to verify that no such $\gamma$ exists. Therefore, our assumption that $\mathrm{HS}_2(\mathcal{L}(P,n))$ is quasi-linear must be false, and hence the result follows.
 \end{proof}

\section{The Resolution}

In this section, we devote to proving the main result of this paper.

We will maintain  Conventions~\ref{setup}. Suppose $|A_1|\geq 2$ and let $a_1$ be the minimal element of $A_1$. Then we have the disjoint union of posets: $$\mathrm{Hom}(P, A_1,\ldots,A_m)=\mathrm{Hom}(P,\{a_1\},\ldots,A_m)\cup \mathrm{Hom}(P,A_1\setminus \{a_1\},\ldots,A_m).$$ Hence, if $\mathfrak{A}$ is a poset ideal of $\mathrm{Hom}(P, A_1,\ldots,A_m)$, then we have the disjoint union $$\mathfrak{A}=\{\mathrm{Hom}(P,\{a_1\},\ldots,A_m)\cap \mathfrak{A}\}\cup \{\mathrm{Hom}(P,A_1\setminus \{a_1\},\ldots,A_m)\cap \mathfrak{A}\}.$$
Denote $\mathfrak{A}_1=\mathrm{Hom}(P,\{a_1\},\ldots,A_m)\cap \mathfrak{A}$ and $\mathfrak{A}_2=\mathrm{Hom}(P, A_1\setminus\{a_1\},\ldots,A_m)\cap \mathfrak{A}$. Then, it is routine to check that $\mathfrak{A}_1$ and $\mathfrak{A}_2$ are poset ideals of $\mathrm{Hom}(P,\{a_1\},\ldots,A_m)$ and $\mathrm{Hom}(P,A_1\setminus\{a_1\},\ldots,A_m)$ respectively. Put $$\mathcal{L}=\mathcal{L}(\mathfrak{A}), \quad \mathcal{J}=\mathcal{L}(\mathfrak{A}_1) \quad \mbox{ and } \quad \mathcal{K}=\mathcal{L}(\mathfrak{A}_2).$$

 \begin{Lemma} \label{Betti} Under the notions as given above, we have $\mathcal{L=J+K}$ is a Betti-splitting. Moreover, $\mathcal{J\cap K}=X_{p_1,a_1}\mathcal{K}.$
\end{Lemma}
\begin{proof} Since $\mathfrak{A}$ is the disjoint union of $\mathfrak{A}_1$ and $\mathfrak{A}_2$, the minimal generating set of $\mathcal{L}$ is the disjoint union of minimal generating sets of $\mathcal{J}$ and $\mathcal{K}$. According to Lemma~\ref{weak},  both $\mathcal{J}$ and $\mathcal{K}$ have an $m$-linear resolution.   It follows from Lemma~\ref{component} that $\mathcal{L=J+K}$ is a Betti-splitting. We now show that $\mathcal{J\cap K}=X_{p_1,a_1}\mathcal{K}.$ Let $X_{p_1,b_1}X_{p_2,b_2}\cdots X_{p_m,b_m}$ be a minimal generator of $\mathcal{K}$. Then, in view of $(b_1,b_2,\ldots,b_m)\in \mathfrak{A}$, we have $(a_1,b_2,\ldots,b_m)\in \mathfrak{A}$. It follow that $X_{p_1,a_1}X_{p_2,b_2}\cdots X_{p_m,b_m}$ is a minimal generator of $\mathcal{J}$. This implies $X_{p_1,a_1}(X_{p_1,b_1}X_{p_2,b_2}\cdots X_{p_m,b_m})\in \mathcal{J\cap K}$ and so $X_{p_1,a_1}\mathcal{K}\subseteq \mathcal{J\cap K}.$

Conversely, let $\alpha=X_{p_1,a_1}\cdots X_{p_m,a_m}$ and $\beta=X_{p_1,b_1}\cdots X_{p_m,b_m}$ be minimal generators of $\mathcal{J}$ and $\mathcal{K}$ respectively. Since $a_1<b_1$,  $X_{p_1,a_1}\beta$ divides the least common multiple $[\alpha, \beta]$ of $\alpha$ and $\beta$. Thus, $[\alpha, \beta]\in X_{p_1,a_1}\mathcal{K}$, and $\mathcal{J\cap K}\subseteq X_{p_1,a_1}\mathcal{K}$, as required.
\end{proof}

We remark that $\mathcal{L=J+K}$ is also a Eliahou-Kervaire splitting,  as introduced in \cite{EK}.
Recall that a map $f:P\rightarrow [n]$ can also be denoted by the $m$-tuple $(f(1),\ldots, f(m))$.
\begin{Construction}\label{con} \em   We construct a complex of multigraded  $R$-modules \begin{equation}\label{resolution11}\mathbb{F}: \xymatrix{0\ar[r] &F_p\ar[r]^{d_p}&F_{p-1}\ar[r]^{d_{p-1}}&\cdots\ar[r]^{d_1}&F_0\ar[r]^{d_0}&\mathcal{L}(\mathfrak{A})\ar[r]&0}
\end{equation} as follows:
\begin{itemize}
\item for $0\leq t\leq p$, the free module $F_t$ has  a basis consisting of   symbols  $[K_1,\ldots, K_m]$ with multidegree $X_{p_1,K_1}\cdots X_{p_m,K_m}$, subject to the following conditions: \begin{enumerate}
                                                              \item $\emptyset \neq K_i\subseteq A_i$ for all $i=1,\ldots,m$;
                                                               \item $|K_1|+\cdots+|K_m|=m+t$;
                                                              \item $\max(K_i)\leq \min(K_j)$ if $p_i<p_j$;
                                                              \item $(\widehat{K_1}, \widehat{K_2}, \ldots, \widehat{K_m})\in \mathfrak{A}$.
                                                            \end{enumerate}
\emph{We will denote $(K_1,\ldots, K_m)\in C_t(\mathfrak{A})$ to indicate that the sets $K_1, \ldots, K_m$ satisfy these four conditions.}

The differential maps $d_t$ are given as follows:

    \item If $(K_1,\cdots,K_m)\in C_0(\mathfrak{A})$, i.e., $[K_1,\cdots,K_m]$ is a base vector of $F_0$,   then $$d_0([K_1,\cdots,K_m])=X_{p_1,K_1}X_{p_2,K_2}\ldots X_{p_m,K_m}\in \mathcal{L}(\mathfrak{A}).$$

\item Given $t\geq 1$ and $(K_1,\cdots,K_m)\in C_t(\mathfrak{A})$, we denote $$\Omega=\{X_{p_i,a}\mid 1\leq i\leq m, |K_i|\geq 2, a\in K_i\}$$ and  $$\overline{\Omega}=\{X_{p_i,a}\mid 1\leq i\leq m,  a\in K_i\}.$$
Note that $|\Omega|=t+|\{i\mid |K_i|\geq 2\}|$. Define an order on $\overline{\Omega}$: $X_{p_i,a}>X_{p_j,b}$ if and only if either $i<j$ or $i=j$ and $a<b$. Let $\sigma(X_{p_i,a},\overline{\Omega})$ be the number of elements in $\overline{\Omega}$ that are larger than $X_{p_i,a}$. Then $$d_t([K_1,\ldots,K_m])=\sum_{X_{p_i,a}\in \Omega}(-1)^{\sigma(X_{p_i,a},\overline{\Omega})}X_{p_i,a}[K_1, \ldots, K_i\setminus\{a\},\ldots, K_m].$$
\end{itemize}

\end{Construction}

\begin{Lemma} \label{complex} The sequence defined above is a complex, i.e., $d_{t-1}d_t=0$ for all $t\geq 1$.
\end{Lemma}
\begin{proof} We first consider the case that $t=1$. Given a basis element $[K_1,\ldots,K_m]$ of $F_1$, there is exactly one $i\in [m]$ such that $|K_i|\geq 2$. Moreover, if $|K_i|\geq 2$, then $|K_i|=2$.  From this observation, it follows easily that  that $d_0d_1([K_1,\ldots,K_m])=0$.
Suppose now that $t>1$ and let  $[K_1,\ldots,K_m]$ be a basis element of $F_t$. It suffices to show that $d_{t-1}d_t([K_1,\ldots,K_m])=0$. Notice  that $d_{t-1}d_t([K_1,\ldots,K_m])$ can be written as an $R$-linear combination of symbols of the following forms: $$[K_1,\ldots,K_i\setminus\{a\},\ldots, K_j\setminus\{b\},\ldots,K_m] \mbox{ for } X_{p_i,a},  X_{p_j,b}\in \Omega \mbox{ with } i<j  $$ and symbols of the form  $$[K_1,\ldots,K_i\setminus\{a,b\},\ldots, K_m] \mbox{ for }  X_{p_i,a},  X_{p_i,b}\in \Omega \mbox{ with } |K_i|\geq 3 \mbox{ and } a<b.$$  Let $X_{p_i,a},  X_{p_j,b}\in \Omega \mbox{ with } i<j $. Then the symbol $[K_1,\ldots,K_i\setminus\{a\},\ldots, K_j\setminus\{b\},\ldots,K_m]$ appears twice in the expansion of $d_{t-1}d_t([K_1,\ldots,K_m])$, with  coefficients  of $$(-1)^{\sigma(X_{p_i,a}, \overline{\Omega})}(-1)^{\sigma(X_{p_j,b}, \overline{\Omega})-1}X_{p_i,a}X_{p_j,b}$$ and  $$(-1)^{\sigma(X_{p_i,a}, \overline{\Omega})}(-1)^{\sigma(X_{p_j,b}, \overline{\Omega})}X_{p_i,a}X_{p_j,b},$$ respectively. Therefore, these two terms cancel each other out in the expansion of $d_{t-1}d_t([K_1,\ldots,K_m])$.  A similar argument can be made for elements  $X_{p_i,a},  X_{p_i,b}\in \Omega$  with  $|K_i|\geq 3$ and $a<b$.  As a result, we conclude that $d_{t-1}d_t([K_1,\ldots,K_m])$ is equal to zero, as needed.
\end{proof}

   Analyzing the proof of Lemma~\ref{complex}, we can see why we should use $ (-1)^{\sigma(X_{p_i,a},\overline{\Omega})}$ instead of $ (-1)^{\sigma(X_{p_i,a}, \Omega)}$  in Construction~\ref{con}.

We  now move forward to prove the main result of this paper. Although the proof is quite lengthy, most parts of it are routine, except for a few steps such as the construction of the comparison  maps.

\begin{Theorem} \label{main4.4} The  sequence in Construction~\ref{con}
is the  minimal multigraded free resolution of $\mathcal{L}(\mathfrak{A})$.
\end{Theorem}
\begin{proof}  Denote $\mathcal{L}(\mathfrak{A})$ as $\mathcal{L}$ for short. We proceed by   induction on the size of $\mathcal{L}$.

 \noindent  {\bf Part 1: initial cases.} If $\mathrm{size}(\mathcal{L})=1$ there is nothing to prove. If $\mathrm{size}(\mathcal{L})=2$, then either $m=1$ and $|A_1|=2$ or $m=2$ and $|A_1|=|A_2|=1$. In the case that  $m=2$ and $|A_1|=|A_2|=1$, $\mathcal{L}$  is a free module of rank one and we are done. In the case that $m=1$ and $|A_1|=2$, we may write $\mathcal{L}=(X_1,X_2)$ and the sequence (\ref{resolution11}) specialized in this case is as follows:
\begin{equation*}\xymatrix{0\ar[r]& R([X_1X_2])\ar[r]^{d_1}& R([X_1])\oplus R([X_2])\ar[r]&\mathcal{L}\ar[r]&0}.\end{equation*} Here, $d_1([X_1X_2])=X_1[X_2]-X_2[X_1]$. One may check easily that it is indeed a multigraded minimal  free resolution of $(X_1,X_2)$.

\noindent  {\bf Part 2: Inductive step.}   Suppose now that $\mathrm{size}(\mathcal{L})\geq 3$. We consider the following two cases.

\noindent  {\bf Part 2.1.}   The first case is when $|A_1|=1$. Let $a_1$ be the unique element of $A_1$. In this case, we may write $$\mathcal{L}=X_{p_1,a_1}\mathcal{L}(P',A'_2,A'_3,\ldots,A'_m; \mathfrak{A}').$$ Here, $P'=\{p_2,\ldots,p_m\}$ is a sub-poset of $P$,   $\mathfrak{A}'=\{(b_2,\ldots,b_m)\mid (a_1,b_2,\ldots,b_m)\in \mathfrak{A}\}$, and for $i=2,\ldots,m$,  $$A'_i=\left\{
                                                                                           \begin{array}{ll}
                                                                                             A_i\cap [a_1,n], & \hbox{$p_1<p_i$;} \\
                                                                                             A_i, & \hbox{otherwise.}
                                                                                           \end{array}
                                                                                         \right.
$$ Note that $\mathfrak{A}'$ is a poset ideal of $\mathrm{Hom}(P', A'_2,A'_3,\ldots,A'_m)$. Denote $\mathcal{L}( \mathfrak{A}')$ as $\mathcal{L}'$. It is straightforward to observe that if  \begin{equation*} \mathbb{F}': \xymatrix{0\ar[r] &F'_p\ar[r]^{d'_p}&F'_{p-1}\ar[r]^{d'_{p-1}}&\cdots\ar[r]^{d'_1}&F'_0\ar[r]^{d'_0}&\mathcal{L'}\ar[r]&0}
\end{equation*}
is a  minimal multigraded  free resolution of $\mathcal{L'}$, then $\mathcal{L}$ admits a minimal multigraded  free resolution of the following form:
\begin{equation*} \mathbb{G}: \xymatrix{0\ar[r] &G_p\ar[r]^{\delta_p}&G_{p-1}\ar[r]^{\delta_{p-1}}&\cdots\ar[r]^{\delta_1}&G_0\ar[r]^{\delta_0}&\mathcal{L}\ar[r]&0,}
\end{equation*} which  is a shift of $\mathbb{F}'$ by $X_{p_1,a_1}$. Specifically, if $F'_i$ has a basis $f'_{i,1},\ldots,f'_{i,k_i}$, then $G_i$ has a corresponding basis $g_{i,1},\ldots,g_{i,k_i}$ where the multidegree of $g_{i,j}$ is the multidegree of $f'_{i,j}$ multiplied by $X_{p_1,a_1}$ for all $i=0,\ldots,p$ and $j=1,\ldots,k_i$. Moreover, the matrices representing  $d'_i$ and $\delta_i$ are identical  for all $i=0,\ldots,p$ with respect to these bases.
Since $\mathrm{size}(\mathcal{L}')\leq \mathrm{size}(\mathcal{L})-1$, by induction,  $\mathbb{F}'$ conforms to the structure  as described in Construction~\ref{con}. Therefore, $\mathbb{G}$ also  follows the same structure, as required.

\noindent  {\bf Part 2.2.}  The other case is when $|A_1|\geq 2$. Let $a_1$ be the minimal element of $A_1$. In this case, $\mathcal{L}$ has the Betti-splitting  as given in Lemma~\ref{Betti}. By induction,  we can assume  the minimal free resolutions of $\mathcal{J}$ and  $\mathcal{K}$ are  given in Construction~\ref{con}, which we denote by $\mathbb{F}^{\mathcal{J}}$ and $\mathbb{F}^{\mathcal{K}}$ respectively. Since $\mathcal{J}\cap \mathcal{K}=X_{p_1,a_1}\mathcal{K}$, the minimal free resolution of $\mathcal{J\cap K}$, denoted by $\mathbb{F}^{\mathcal{J\cap K}} $ is the same as $\mathbb{F}^{\mathcal{K}}$ except a shift by $X_{p_1,a_1}$. Let $\mathbb{F}^{\mathcal{L}}$ be the resolution obtained by applying mapping cone to the Betti-splitting $\mathcal{L}=\mathcal{J}+\mathcal{K}$. Then $\mathbb{F}^{\mathcal{L}}$ is the minimal multigraded free resolution of $\mathcal{L}$. It remains to show that $\mathbb{F}^{\mathcal{L}}$  can be identified with $\mathbb{F}$. We achieve this goal in the following steps.

Step 1: We show that the free modules in the resolution  $\mathbb{F}^{\mathcal{L}}$ have the bases described  in Construction~\ref{con}. By Subsection~\ref{mapping}, we have $F^{\mathcal{L}}_t=F^{\mathcal{J\cap K}}_{t-1}\oplus F^{\mathcal{J}}_t\oplus F^{\mathcal{K}}_t.$
The basis of $F^{\mathcal{J}}_t$ and  $F^{\mathcal{K}}_t$  are the following sets, respectively:
\begin{equation*}\begin{split}
&\mathfrak{B}_1:=\{[K_1,\ldots, K_m]\mid K_1=\{a_1\},  (K_1, \ldots, K_m)\in C_t(\mathfrak{A})\},\\
&\mathfrak{B}_2:=\{[K_1,\ldots, K_m]\mid  a_1\notin K_1,  (K_1, \ldots, K_m)\in C_t(\mathfrak{A})  \}.
\end{split}
\end{equation*}
The basis of $F^{\mathcal{J\cap K}}_{t-1}$ is a shift of basis of $F^{\mathcal{K}}_{t-1}$ by $X_{p_1,a_1}$, so it is
\begin{equation*}\begin{split}\mathfrak{B}_3:=\{X_{p_1,a_1}[K_1,\ldots, K_m]\mid  a_1\notin K_1,  (K_1, \ldots, K_m)\in C_{t-1}(\mathfrak{A})\},
\end{split}
\end{equation*}
By identifying  $X_{p_1,a_1}[K_1,\ldots, K_m]$ with $[K_1\cup\{a_1\}, K_2, \ldots, K_m]$, we may also write
$$\mathfrak{B}_3:=\{[K_1,\ldots, K_m]\mid  \{a_1\}\subsetneqq K_1,  (K_1, \ldots, K_m)\in C_{t}(\mathfrak{A})\}$$
On the other hand,  $F_t$ has a basis
$$\mathfrak{C}:=\{[K_1,\ldots, K_m]\mid   (K_1, \ldots, K_m)\in C_{t}(\mathfrak{A})\}$$
Therefore, it is easy to see that $\mathfrak{B}_1, \mathfrak{B}_2, \mathfrak{B}_3$ form a partition of $\mathfrak{C}$, implying $F_t^{\mathcal{L}}$ and $F_t$ share the same basis.

Step 2:  Let $f_{-1}:u\in \mathcal{J\cap K}\mapsto(-u,u)\in \mathcal{J\oplus K}$. In order to apply the mapping cone, we need to construct the comparison maps $f_t, t=0,1,\ldots$ such that the following diagram
	$$
	\xymatrix{
		\displaystyle
		\mathbb{F}^{\mathcal{J\cap K}}:\cdots\ar[r] &F_2^{\mathcal{J\cap K}} \ar[d]_{f_2}\ar[r]^{d_2^{\mathcal{J\cap K}}} & F_1^{\mathcal{J\cap K}} \ar[d]_{f_1}\ar[r]^{d_1^{\mathcal{J\cap K}}} & F_0^{\mathcal{J\cap K}}\ar[d]_{f_0}\ar[r]^{d_0^{\mathcal{J\cap K}}}& \mathcal{J\cap K}\ar[d]^{f_{-1}}\ar[r]&0\\
		\mathbb{F}^{\mathcal{J}}\oplus \mathbb{F}^{\mathcal{K}}:\cdots\ar[r] &F_2^{\mathcal{J}}\oplus F_2^{\mathcal{K}} \ar[r]_{d_2^{\mathcal{J}}\oplus d_2^{\mathcal{K}}} & F_1^{\mathcal{J}}\oplus F_1^{\mathcal{K}} \ar[r]_{d_1^{\mathcal{J}}\oplus d_1^{\mathcal{K}}} & F_0^{\mathcal{J}}\oplus F_0^{\mathcal{K}}\ar[r]_{d_0^{\mathcal{J}}\oplus d_0^{\mathcal{K}}}&\mathcal{J}\oplus \mathcal{K}\ar[r]&0
	}
	$$
commutative. Let $t\geq 0$ and let
 $[K_1,\ldots,K_m]$ be a basis element  of $F^{\mathcal{J\cap K}}_t$. Then $|K_1|+\cdots+|K_m|=m+t+1$, $a_1\in K_1$ and $K_1\setminus\{a_1\}\neq \emptyset.$ It should be noted that $$d_t^{\mathcal{J\cap K}}([K_1,\ldots,K_m])=X_{p_1,a_1}d_t^{\mathcal{K}}([K_1\setminus \{a_1\},K_2,\ldots,K_m]).$$
Now, we define $f_t$ by $f_t([K_1,\ldots,K_m])=(f_{t,\mathcal{J}}([K_1,\ldots,K_m]), f_{t,\mathcal{K}}([K_1,\ldots,K_m]))$, where
$$f_{t,\mathcal{J}}([K_1,\ldots,K_m])=\left\{
                                            \begin{array}{ll}
                                             -X_{p_1,a_2}[\{a_1\},K_2,\ldots,K_m], & \hbox{$|K_1|=2, K_1=\{a_1,a_2\}$;} \\
                                              0, & \hbox{$|K_1|\geq 3$.}
                                            \end{array}
                                          \right.\in \mathbb{F}^{\mathcal{J}}_t,$$
                                          and
$$f_{t,\mathcal{K}}([K_1,\ldots,K_m])=X_{p_1,a_1}[K_1\setminus \{a_1\},K_2,\ldots,K_m]\in \mathbb{F}^{\mathcal{K}}_t.$$
  We claim that $(d_t^{\mathcal{J}}\oplus d_t^{\mathcal{K}})f_t([K_1,\ldots,K_m])=f_{t-1}d_t^{\mathcal{J\cap K}}([K_1,\ldots,K_m])$.  If we can prove this claim, then  $f_t, t=0,1,\ldots$ are comparison maps. We establish  this claim in the following two cases.

Case 1: $|K_1|=2$ and $K_1=\{a_1,a_2\}$.   In this case, \begin{equation*}\begin{split} (d_t^{\mathcal{J}}\oplus d_t^{\mathcal{K}})f_t([K_1,\ldots,K_m])=(-X_{p_1,a_2}d_t^{\mathcal{J}}([\{a_1\},\ldots,K_m]), X_{p_1,a_1}d_t^{\mathcal{K}}([\{a_2\}, K_2,\ldots,K_m])).
\end{split}\end{equation*}
Since \begin{equation*}\begin{split} f_{t-1,\mathcal{K}}d_t^{\mathcal{J\cap K}}([K_1,\ldots,K_m])&=f_{t-1,\mathcal{K}}(X_{p_1,a_1}d_t^{\mathcal{K}}([\{a_2\}, K_2,\ldots,K_m]))\\&=X_{p_1,a_1}d_t^{\mathcal{K}}([\{a_2\}, K_2,\ldots,K_m],\end{split}\end{equation*}
 it suffices to show that $$f_{t-1,\mathcal{J}}(X_{p_1,a_1}d_t^{\mathcal{K}}([\{a_2\},K_2,\ldots,K_m]))=-X_{p_1,a_2}d_t^{\mathcal{J}}([\{a_1\},\ldots,K_m]).$$
 Let $\Omega_1=\{X_{p_i,a}\mid 2\leq i\leq m, |K_i|\geq 2, a\in K_i\}$ and $\overline{\Omega_1}=\{X_{p_i,a}\mid 2\leq i\leq m,  a\in K_i\}$. Then
 \begin{equation*}\begin{split}&f_{t-1,\mathcal{J}}(X_{p_1,a_1}d_t^{\mathcal{K}}([\{a_2\},K_2,\ldots,K_m]))
 \\& =f_{t-1,\mathcal{J}}(X_{p_1,a_1}\sum_{X_{p_i,a}\in \Omega_1}(-1)^{\sigma(X_{p_i,a},\overline{\Omega_1})+1}X_{p_i,a}[\{a_2\},K_2,\ldots,K_i\setminus\{a\},\ldots,K_m])
 \\& =f_{t-1,\mathcal{J}}(\sum_{X_{p_i,a}\in \Omega_1}(-1)^{\sigma(X_{p_i,a},\overline{\Omega_1})+1}X_{p_i,a}[\{a_1,a_2\},K_2,\ldots,K_i\setminus\{a\},\ldots,K_m])
 \\&=-X_{p_1,a_2}\sum_{X_{p_i,a}\in \Omega_1}(-1)^{\sigma(X_{p_i,a},\overline{\Omega_1})+1}X_{p_i,a}[\{a_1\},K_2,\ldots,K_i\setminus\{a\},\ldots,K_m])
 \\&=-X_{p_1,a_2}d_t^{\mathcal{J}}([\{a_1\},\ldots,K_m]).
 \end{split}\end{equation*}

 Case 2: $|K_1|\geq 3$.  In this case, \begin{equation*}\begin{split}  (d_t^{\mathcal{J}}\oplus d_t^{\mathcal{K}})f_t([K_1,\ldots,K_m])=(0,X_{p_1,a_1}d_t^{\mathcal{K}}([K_1\setminus \{a_1\}, K_2,\ldots,K_m]))
\end{split}\end{equation*}
 and $$f_{t-1}d_t^{\mathcal{J\cap K}}([K_1,\ldots,K_m])=f_{t-1}(X_{p_1,a_1}d_t^{\mathcal{K}}([K_1\setminus\{a_1\}, K_2,\ldots,K_m])).$$
 Since $f_{t-1,\mathcal{K}}(X_{p_1,a_1}d_t^{\mathcal{K}}([K_1\setminus\{a_1\}, K_2,\ldots,K_m]))=X_{p_1,a_1}d_t^{\mathcal{K}}([K_1\setminus \{a_1\}, K_2,\ldots,K_m]$, it suffices to prove $f_{t-1,\mathcal{J}}(X_{p_1,a_1}d_t^{\mathcal{K}}([K_1\setminus\{a_1\}, K_2,\ldots,K_m]))=0$. If $|K_1|\geq 4$, then we are done by the definition of $f_{t-1,\mathcal{J}}$ as well as the the fact that each term in the expansion
of $d_t^{\mathcal{K}}([K_1\setminus\{a_1\}, K_2,\ldots,K_m]))$ has the form $u[K'_1,\ldots, K'_m]$ such that $|K'_1|\geq 2.$ Here $u$ is a monomial.
 Suppose now that $|K_1|=3$ and $K_1=\{a_1,a_2,a_3\}$ with $a_1<a_2<a_3$. Put $\Omega_2=\{X_{p_i,a}\mid 2\leq i\leq m, |K_i|\geq 2, a\in K_i\}$ and $\overline{\Omega_2}=\{X_{p_i,a}\mid 2\leq i\leq m,  a\in K_i\}$. Then
  \begin{equation*}\begin{split}&f_{t-1,\mathcal{J}}(X_{p_1,a_1}d_t^{\mathcal{K}}([\{a_2,a_3\},K_2,\ldots,K_m]))\\
  &=f_{t-1,\mathcal{J}}(X_{p_1,a_1}(X_{p_1,a_2}[\{a_3\},K_2,\ldots,K_m]-X_{p_1,a_3}[\{a_2\},K_2,\ldots,K_m]))\\
  &+f_{t-1,\mathcal{J}}(X_{p_1,a_1}\sum_{X_{p_i,a}\in \Omega_2}(-1)^{\sigma(X_{p_i,a},\overline{\Omega_2})}X_{p_i,a}[\{a_2,a_3\},K_2,\ldots,K_i\setminus\{a\},\ldots,K_m]).
 \end{split}\end{equation*}
 Note that $f_{t-1,\mathcal{J}}(X_{p_1,a_1}(X_{p_1,a_2}[\{a_3\},K_2,\ldots,K_m])=X_{p_1,a_2}X_{p_1,a_3}[\{a_1\},K_2,\ldots,K_m]$
 and  $f_{t-1,\mathcal{J}}(X_{p_1,a_1}(X_{p_1,a_3}[\{a_2\},K_2,\ldots,K_m])=X_{p_1,a_2}X_{p_1,a_3}[\{a_1\},K_2,\ldots,K_m]$, it follows that $$f_{t-1,\mathcal{J}}(X_{p_1,a_1}(X_{p_1,a_2}[\{a_3\},K_2,\ldots,K_m]-X_{p_1,a_3}[\{a_2\},K_2,\ldots,K_m]))=0.$$
  It is evident  that $f_{t-1,\mathcal{J}}([\{a_1,a_2,a_3\},K_2,\ldots,K_i\setminus\{a\},\ldots,K_m])=0$ for all $X_{p_i,a}\in \Omega_2$. Thus, the claim has been proved.

Step 3: It remains to show that $d^{\mathcal{L}}_t$, the $t$-th differential of $\mathbb{F}^{\mathcal{L}}$,  is identical to  $d_t$ as defined in Construction~\ref{con} for all $t\geq 0$. Note that by definition, we have $d^{\mathcal{L}}_t=\begin{pmatrix}-d_{t-1}^{\mathcal{J\cap K}}&0\\f_{t-1}& d_{t}^{\mathcal{J}}\oplus d_{t}^{\mathcal{K}}\end{pmatrix}.$ Let $\alpha=[K_1,\ldots,K_m]$ be a  basis element of $F_t^{\mathcal{L}}$. If $\alpha\in \mathfrak{B}_1$, then $d_t^{\mathcal{L}}(\alpha)=d_t^{\mathcal{J}}(\alpha)=d_t(\alpha)$, and if $\alpha\in \mathfrak{B}_2$, then $d_t^{\mathcal{L}}(\alpha)=d_t^{\mathcal{K}}(\alpha)=d_t(\alpha)$. Now assume $\alpha=[K_1,\ldots,K_m]\in \mathfrak{B}_3$. Note that $a_1\in K_1$ and $K_1\setminus \{a_1\}\neq \emptyset$. Write $K_1\setminus \{a_1\}=\{a_2,\ldots,a_{k_1}\}$, where $a_2<\cdots<a_{k_1}$.  Put $$\Omega=\{X_{p_i,a}\mid 2\leq i\leq m, |K_i|\geq 2, a\in K_i\}$$ and $$\overline{\Omega}=\{X_{p_i,a}\mid 2\leq i\leq m,  a\in K_i\}.$$ We consider the following cases.

\noindent{\it Case 1:}  If $k_1=2$, i.e., $K_1=\{a_1,a_2\}$, then \begin{equation*}\begin{split} d_t^{\mathcal{L}}(\alpha)&=-d_{t-1}^{\mathcal{J\cap K}}(\alpha)+f_{t-1}(\alpha)
\\&=-X_{p_1,a_1}d_{t-1}^{\mathcal{K}}([\{a_2\},K_2,\ldots,K_m])-X_{p_1,a_2}[\{a_1\},\ldots,K_m]+X_{p_1,a_1}[\{a_2\},\ldots,K_m]
\\&=\sum_{X_{p_i,a}\in \Omega}(-1)^{\sigma(X_{p_i,a},\overline{\Omega})+2}X_{p_i,a}X_{p_1,a_1}[\{a_2\},K_2,\ldots,K_i\setminus\{a\},\ldots,K_m]
\\&\qquad -X_{p_1,a_2}[\{a_1\},\ldots,K_m] +X_{p_1,a_1}[\{a_2\},\ldots,K_m]
\\&=\sum_{X_{p_i,a}\in \Omega}(-1)^{\sigma(X_{p_i,a},\overline{\Omega})+2}X_{p_i,a}[\{a_1,a_2\},K_2,\ldots,K_i\setminus\{a\},\ldots,K_m]
\\&\qquad -X_{p_1,a_2}[\{a_1\},\ldots,K_m] +X_{p_1,a_1}[\{a_2\},\ldots,K_m].
\end{split}\end{equation*}
Put $\Omega'=\Omega\cup \{X_{p_1,a_1},X_{p_1,a_2}\}$ and $\overline{\Omega'}=\overline{\Omega}\cup \{X_{p_1,a_1},X_{p_1,a_2}\}$. Then
\begin{equation*}\begin{split} d_t(\alpha)&=\sum_{X_{p_i,a}\in \Omega'}(-1)^{\sigma(X_{p_i,a},\overline{\Omega'})}[K_1,K_2,\ldots,K_i\setminus\{a\},\ldots,K_m]
\\&=X_{p_1,a_1}[\{a_2\},K_2,\ldots,K_m]-X_{p_1,a_2}[\{a_1\},K_2,\ldots,K_m]
\\&\quad+ \sum_{X_{p_i,a}\in \Omega}(-1)^{\sigma(X_{p_i,a},\overline{\Omega})+2}X_{p_i,a}[K_1,K_2,\ldots,K_i\setminus\{a\},\ldots,K_m]
\end{split}\end{equation*}
Therefore, we have $d_t^{\mathcal{L}}(\alpha)=d_t(\alpha)$, as required.

\noindent{\it Case 2:} If $k_1\geq 3$, we put $$\Omega''=\Omega\cup \{X_{p_1,a_i}\mid i=1,\ldots,k_1\}$$ and $$\overline{\Omega''}=\overline{\Omega}\cup \{X_{p_1,a_i}\mid i=1,\ldots,k_1\},$$ then
\begin{equation*}\begin{split} d_t^{\mathcal{L}}(\alpha)&=-d_{t-1}^{\mathcal{J\cap K}}(\alpha)+f_{t-1}(\alpha)
\\&=-X_{p_1,a_1}d_{t-1}^{\mathcal{K}}([K_1\setminus\{a_1\},K_2,\ldots,K_m])+X_{p_1,a_1}[K_1\setminus \{a_1\},\ldots,K_m]\\&=X_{p_1,a_1}[K_1\setminus \{a_1\},\ldots,K_m]-\sum_{i=2}^{k_1}(-1)^iX_{p_1,a_1}X_{p_1,a_i}[K_1\setminus\{a_1,a_i\},K_2,\ldots,K_m]
\\&-\sum_{X_{p_i,a}\in \Omega}(-1)^{\sigma(X_{p_i,a},\overline{\Omega})+(k_1-1)}X_{p_1,a_1}X_{p_i,a}[K_1\setminus \{a_1\} ,K_2,\ldots,K_i\setminus\{a\},\ldots,K_m]\\&=X_{p_1,a_1}[K_1\setminus \{a_1\},\ldots,K_m]-\sum_{i=2}^{k_1}(-1)^iX_{p_1,a_i}[K_1\setminus\{a_i\},K_2,\ldots,K_m]
\\&+\sum_{X_{p_i,a}\in \Omega}(-1)^{\sigma(X_{p_i,a},\overline{\Omega})+k_1}X_{p_i,a}[K_1,K_2,\ldots,K_i\setminus\{a\},\ldots,K_m]
\\&=\sum_{X_{p_i,a}\in \Omega''}(-1)^{\sigma(X_{p_i,a},\overline{\Omega''})}X_{p_i,a}[K_1,K_2,\ldots,K_i\setminus\{a\},\ldots,K_m]\\&=d_{t}(\alpha).
\end{split}\end{equation*}
This completes the proof.\end{proof}

Since the resolution presented in Theorem~\ref{main4.4} is completely determined by the combinatorial data of $P, \mathcal{A} $ and so on, we can immediately obtain the following result.

\begin{Corollary} The projective dimension and multigraded Betti numbers of a generalized co-letterplace ideal are independent of the underlying field $\mathbf{K}.$
\end{Corollary}

Let $I$ be a monomial ideal having linear quotients. If $I$ admits a regular decomposition function,  an explicit multigraded minimal free resolution for $I$ was provided in \cite{HT02}.  Furthermore, it was proved in \cite{DM} and in \cite{Good} independently  that this free resolution is supported by a regular cell complex.  It is natural to ask the following question.
\vspace{2mm}

\noindent {\bf \large Question:} Could we find  a regular cell complex that supports the resolution presented in Theorem~\ref{main4.4}?
\vspace{2mm}

\section{Powers}

In this section we   aim to demonstrate that if $\mathfrak{A}=\mathrm{Hom}(P,A_1,\ldots,A_m)$ then for all integers $k\geq 1$, the ring $R/\mathcal{L}(\mathfrak{A})^k$ can be represented as the quotient ring of $T/J$ modulo a regular sequence of variable differences. Here, $J$ is another generalized co-letterplace ideal residing in a larger ring $T$ compared to $R$.   Consequently, we can obtain an explicit multigraded minimal free resolution for all powers of such generalized co-letterplace ideals. It is worth to note that if $|A_i|\leq 2$ for all $i=1,\ldots,m$ then $J$ is the polarization of $\mathcal{L}(\mathfrak{A})^k$, but  this is not the case in general. We begin by introducing a new notation.

\begin{Notation} Let $P,\mathcal{A},\mathfrak{A}$ be as defined in Conventions~\ref{setup} and let $k\geq 1$. We denote  the ideal of $R$ generated by all the monomials $U_{f_1}U_{f_2}\cdots U_{f_k}$ with $f_1\leq f_2\leq \cdots \leq f_k\in \mathfrak{A}$ as $\mathcal{L}(\mathfrak{A};k)$.
\end{Notation}

It is clear that  $\mathcal{L}(\mathfrak{A};k)\subseteq \mathcal{L}(\mathfrak{A})^k$, but $\mathcal{L}(\mathfrak{A};k)\neq \mathcal{L}(\mathfrak{A})^k$ in general.

\begin{Lemma} For all $k\geq 1$, $\mathcal{L}(\mathfrak{A};k)$ is weakly polymatroidal.
\end{Lemma}

\begin{proof} Let $\alpha$ and $\beta$ be distinct minimal generator of  $\mathcal{L}(\mathfrak{A};k)$. We may write $\alpha=U_{f_1}U_{f_2}\cdots U_{f_k}$ and $\beta=U_{g_1}U_{g_2}\cdots U_{g_k}$, where $f_1\leq f_2\leq \cdots \leq f_k\in \mathfrak{A}$ and $g_1\leq g_2\leq \cdots \leq g_k\in \mathfrak{A}$. Let $X_{p_s,a}$ be the maximal variable $X$ such that $\deg_X(\alpha)>\deg_X(\beta)$. Then $f_i(p_j)=g_i(p_j)$ for $i=1,\ldots,k$ and $j=1,\ldots, s-1$, and there is $t\in [k]$ such that $f_i(p_s)=g_i(p_s)$ for $i=1,\ldots, t-1$ and $a=f_t(p_s)<g_t(p_s)$. Define   a map: $h:P\rightarrow [n]$ as follows: $$h(p_i)=\left\{
                       \begin{array}{ll}
                         a, & \hbox{$i=s$;} \\
                         g_t(p_i), & \hbox{$i\neq s$.}
                       \end{array}
                     \right.$$
                     Then, it is routine to verify that $h\in \mathfrak{A}$ and $g_{t-1}\leq h\leq g_t$. Put $b=g_t(p_s)$. Then $\frac{\beta X_{p_s,a}}{X_{p_s,b}}=U_{g_1}\cdots U_{g_{t-1}}U_hU_{g_{t+1}}\cdots U_k\in L(\mathfrak{A};k)$. This completes the proof.
\end{proof}

Let $P, \mathcal{A}, \mathfrak{A}$ be the same as in Conventions~\ref{setup}. Let $k\geq 1$.  We define $P^k, \mathcal{A}^k, \mathfrak{A}^k$ as follows.

\begin{itemize}
  \item $P^k$ is a poset, as a set $P^{k}=\{p_{i,j}\mid i=1,\ldots,m, j=1,\ldots,k\}$. Furthermore, $$p_{i_1,j_1}\leq p_{i_2,j_2}\Longleftrightarrow p_{i_1}\leq p_{i_2} \mbox{ and } j_1\leq j_2.$$ One may observe that $P^k$  is isomorphic to the direct product $P\times [k]$.

  \item The map $\mathcal{A}^{k}$ on $P^{k}$ is defined as follows: for $i=1,\ldots,m$ and $j=1,\ldots,k$, we assign  $\mathcal{A}^{k}(p_{i,j})=A_i$.
  \item Let $\mathcal{M}(\mathfrak{A})$  denote the set of all maximal elements of $\mathfrak{A}.$ We define $\mathfrak{A}^{k}$ to be the poset ideal of $\mathrm{Hom}(P^k,\mathcal{A}^k)$ generated by $\{\overline{g}\mid g\in \mathcal{M}(\mathfrak{A})\}$. Here, for $g\in \mathfrak{A}$, $\overline{g}$ is given by $\overline{g}(p_{i,j})=g(p_i)$ for $i=1,\ldots,m$ and $j=1,\ldots,k$.
\end{itemize}

  A technique  for generating a  regular sequence for a quotient ring of a polynomial ring modulo  a co-letterplace ideal is presented in \cite[Theorems 5.6]{FGH}. However, it is unclear whether the regular sequence presented in the following theorem  can be produced in a similar technique.

\begin{Theorem} \label{power} Keep the notions above. Let $T$ be the polynomial ring $$\mathbf{K}[X_{p_{i,j},a}\mid 1\leq i\leq m, 1\leq j\leq k, a\in A_i].$$ Then
 $R/\mathcal{L}(\mathfrak{A};k)$ is the quotient ring of    $T/\mathcal{L}(P^k, \mathcal{A}^k; \mathfrak{A}^{k})$ by a regular sequence of variable differences.
 \end{Theorem}
\begin{proof} Put $\mathfrak{A}^{(k)}=\{(f_1,\ldots,f_k)\mid f_i\in \mathfrak{A} \mbox{ for } i=1,\ldots, k \mbox{ and } f_1\leq f_2\leq \cdots \leq f_k\}$. Let $(f_1,\ldots,f_k)\in \mathfrak{A}^{(k)}$, we define a map $f$ on $P^k$ by
$$f(p_{i,j})=f_j(p_i) \mbox{ for } i=1,\ldots,m, j=1,\ldots,k.$$
One easily checks that $f$ belongs to $\mathrm{Hom}(P^k, \mathcal{A}^k)$. Take $g\in \mathcal{M}(\mathfrak{A})$ such that $f_k\leq g$, then $f(p_{i,j})\leq f_k(p_i)\leq g(p_i)=\overline{g}(p_{i,j})$  for all possible $i,j$. This implies $f\in \mathfrak{A}^k$.
Conversely, let $f\in \mathfrak{A}^k$, we define $f_j(p_i)=f(p_{i,j})$ for $i=1,\ldots,m, j=1,\ldots,k$. Then it is easy to check that $(f_1,\ldots,f_k)\in \mathfrak{A}^{(k)}$. Thus, we establish a bijection between $\mathfrak{A}^{(k)}$ and $\mathfrak{A}^k$. In what follows we will identify $\mathfrak{A}^{(k)}$ with $\mathfrak{A}^k$ and  write $f=(f_1,\ldots,f_k)$ when $f$ corresponds to  $(f_1,\ldots,f_k)$.

To simplify the notation, we will use $X_{i,a}^{(j)}$ to represent $X_{p_{i,j},a}$ and $f(i)$ to represent $f(p_i)$ in the rest part of this  proof. Under these notions, for a given $f=(f_1,\ldots,f_k)\in \mathfrak{A}^k$, we have $$U_f=\prod_{i=1}^m\prod_{j=1}^kX_{p_{i,j},f(p_{i,j})}=\prod_{i=1}^m X_{i,f_1(i)}^{(1)}\cdots X_{i,f_k(i)}^{(k)}.$$

Let us fix $i\in [m]$, $a\in A_i$ and $1\leq s<t\leq k$. We now claim that $X_{i,a}^{(s)}-X_{i,a}^{(t)}$ is  regular on   $T/\mathcal{L}(\mathfrak{A}^{k})$.  To prove this, we assume on the contrary that $X_{i,a}^{(s)}-X_{i,a}^{(t)}$ is a zero-divisor on $T/\mathcal{L}(\mathfrak{A}^{k})$. Then, there is a monomial $W$ such that $W\notin \mathcal{L}(\mathfrak{A}^{k})$, while $X_{i,a}^{(s)}W$ and $X_{i,a}^{(t)}W$ both belong to $\mathcal{L}(\mathfrak{A}^{k})$.
We may write $$X_{i,a}^{^{(s)}}W=U_fW_1 \mbox{ and }   X_{i,a}^{(t)}W=U_gW_2$$ for some $f,g\in \mathfrak{A}^k$. It is easy to see that  both $X_{i,a}^{^{(s)}}$ divides $U_f$ and $X_{i,a}^{^{(t)}}$ divide $U_g$. Moreover, $$[U_f/X_{i,a}^{^{(s)}}, U_g/X_{i,a}^{(t)}] \mbox{ divides } W.$$  Here, $[U,V]$ denotes the least common multiple of monomials $U$ and $V$. Let $f=(f_1,\ldots,f_k)$ and $g=(g_1,\ldots,g_k)$. Then $$U_f=\prod_{j=1}^m X_{j,f_1(j)}^{(1)}\cdots X_{j,f_k(j)}^{(k)} \mbox{ and } U_g=\prod_{j=1}^m X_{j,g_1(j)}^{(1)}\cdots X_{j,g_k(j)}^{(k)}.$$
It follows that \begin{equation*}\begin{split}[U_f/X_{i,a}^{^{(s)}}, U_g/X_{i,a}^{(t)}]=&\prod_{j\in [m]\setminus\{ i\}} [X_{j,f_1(j)}^{(1)}, X_{j,g_1(j)}^{(1)}]\cdots [X_{j,f_k(j)}^{(k)}, X_{j,g_k(j)}^{(k)}]\\&\cdot \prod_{\ell\in [k] \setminus \{s,t\}}[X_{i,f_{\ell}(i)}^{(\ell)}, X_{i,g_{\ell}(i)}^{(\ell)}]\cdot X_{i,g_s(i)}^{^{(s)}}\cdot X_{i,f_t(i)}^{(t)}.\end{split}\end{equation*}
On the other hand, since $f_s(i)=a$ and $g_t(i)=a$, it follows that  $g_s(i)\leq g_t(i)= f_s(i)$, and thus $\min\{f_s(i),g_s(i)\}=g_s(i)$.

To proceed, we require a new notion. For $f,g\in \mathfrak{A}$, let $f\wedge g$ denote the map from $P$ to $[n]$ defined as $(f\wedge g)(i)=\min\{f(i), g(i)\}$ for all $i=1,\ldots,m$. It can be easily verified that $f\wedge g$ also belongs to $\mathfrak{A}$.

 Put $$h=(f_1\wedge g_1,\ldots,f_s\wedge g_s, f_{s+1}, f_{s+2},\ldots, f_k).$$
Then, $h$ belongs to $\mathfrak{A}^{(k)}=\mathfrak{A}^k$. Furthermore, since $(f_s\wedge g_s)(i)=g_s(i)$, it follows that  $U_h$ divides $[U_f/X_{i,a}^{^{(s)}}, U_g/X_{i,a}^{(t)}]$ and therefore it also divides $W$.  This leads to a contradiction. Hence, the claim has been proven.

 We next show   that variable differences $$X_{i,a}^{(s)}-X_{i,a}^{(s+1)}, \  i=1,\ldots,m,\ a\in A_i,\ 1\leq s\leq k-1$$ form  a regular sequence on $T/\mathcal{L}(\mathfrak{A}^k)$.  Note that, by e.g. \cite[Theorem A.3.4]{HH}, the order of the elements of this sequence does't matter.

Fix $i\in [m], a\in A_i$, then, we have $X_{i,a}^{(1)}-X_{i,a}^{(2)}$ is regular on $\mathcal{L}(\mathfrak{A}^k)$ by the claim. Furthermore,   modulo the relation $X_{i,a}^{(1)}-X_{i,a}^{(2)}$, $\mathcal{L}(\mathfrak{A}^k)$ is isomorphic to the monomial ideal $\mathcal{L}_1$, where  the  generators of $\mathcal{L}_1$ are the same as  $\mathcal{L}(\mathfrak{A}^k)$ except  replacing   $X_{i,a}^{(1)}$  with $X_{i,a}^{(2)}$ in each minimal generator of $\mathcal{L}(\mathfrak{A}^k)$  that contains $X_{i,a}^{(1)}$. Then, taking the same method in proving the claim, we can prove   $X_{i,a}^{(2)}-X_{i,a}^{(3)}$ is regular on $\mathcal{L}_1$. To be precise, for each $f\in \mathfrak{A}^k$, we let $$U^1_f=\left\{
                                                                                                             \begin{array}{ll}
                                                                                                               \frac{U_fX_{i,a}^{(2)}}{X_{i,a}^{(1)}}, & \hbox{$X_{i,a}^{(1)} \mbox{ divides } U_f$;} \\
                                                                                                               U_f, & \hbox{otherwise.}
                                                                                                             \end{array}
                                                                                                           \right.
$$
Then, $\mathcal{L}_1=(U^1_f\mid f\in \mathfrak{A}^k)$. By assuming the contrary, there is a monomial $W$ that does not belong to $\mathcal{L}_1$, but satisfies the conditions that $U^1_f$ divides $X_{i,a}^{(2)}W$ and $U^1_g$  divides $X_{i,a}^{(3)}W$ for some $f,g\in \mathfrak{A}^k$. Then, we could check that $U^1_h$ divides $W$, which leads a contradiction. Here, $h$ is constructed as in the proof of the above claim. Continuing in this way, we can see that the sequence given above is indeed a regular sequence.

Finally, it can be easily seen that  $\mathcal{L}(\mathfrak{A}^k)$, modulo the above  sequence, is isomorphic to $\mathcal{L}(\mathfrak{A};k)$. This completes the proof.
\end{proof}

We have known that $\mathrm{Hom}(P,A_1,\ldots,A_m)$ is a poset naturally. In fact, it is also a lattice, where $f\vee g$ and $f\wedge g$ are given by $(f\vee g)(p_i)=\max\{f(p_i),g(p_i)\}$ and $(f\wedge  g)(p_i)=\min\{f(p_i),g(p_i)\}$ for all $i\in [m]$, respectively.  We now show that $\mathcal{L}(\mathfrak{A};k)=\mathcal{L}(\mathfrak{A})^k$ for some (equivalently for all) $k\geq 2$ if and only if $\mathfrak{A}$ is a sublattice of $\mathrm{Hom}(P,A_1,\ldots,A_m)$.

The proof can be simplified by using the concept of multi-sets. Recall a {\it multi-set} is a mathematical structure that allows duplicate elements. Unlike a set, which only allows unique elements, a multi-set can contain multiple instances of the same element.
For example, a multi-set $\{1, 2, 2, 3, 3, 3\}$ contains three instances of the element 2 and three instances of the element 3. The {\it size} of a multi-set $A$, denoted by $|A|$,  is the number of elements of $A$. Thus, $|\{1, 2, 2, 3, 3, 3\}|=6$. Let $A$ be a set. We  denote the multi-set that contains $k$ copies of all elements of $A$ as $\{A\}_{k}$.  For example, if $A=\{1,2\}$, then $\{A\}_2=\{1,1,2,2\}$. In this section we only consider sub-multi-sets of $\{[n]\}_k$ for some $k\geq 1$.

  Given a multi-set $A\subseteq \{[n]\}_k$, we define $N_{\ell}(A)$ to be the $\ell$-th minimal element of $A$ if $1\leq \ell \leq |A|$ and to be $n$ if $\ell>|A|$. In other words, if we write $A=\{a_1,\ldots,a_r\}$ with $a_1\leq a_2\leq \cdots\leq a_r$ then $N_{\ell}(A)=a_{\ell}$ if $\ell\leq r$ and $N_{\ell}(A)=n$ if $\ell>r$.

\begin{Lemma} \label{N} Let $A,B$ be multi-sets contained in $\{[n]\}_k$ such that there is a bijective map $\phi: A\rightarrow B$ satisfying $a\leq \phi(a)$ for all $a\in A$. Then $N_{\ell}(A)\leq N_{\ell}(B)$ for all $\ell=1,2,\ldots,|A|$.
\end{Lemma}

\begin{Proposition}  \label{equality}  Under Conventions~\ref{setup},  the following statements are equivalent:
\begin{enumerate}
  \item $\mathcal{L}(\mathfrak{A};k)=\mathcal{L}(\mathfrak{A})^k$ for some $k\geq 2$;
  \item  $\mathfrak{A}$ is  a sublattice of $\mathrm{Hom}(P,A_1,\ldots, A_m)$;
  \item  $\mathfrak{A}$ has a unique maximal element;
    \item There exist $B_i\subseteq A_i$ for $i=1,\ldots,m$ such that $\mathfrak{A}=\mathrm{Hom}(P,B_1,\ldots, B_m)$;
  \item  $\mathcal{L}(\mathfrak{A};k)=\mathcal{L}(\mathfrak{A})^k$ for all $k\geq 1$.
\end{enumerate}
\end{Proposition}

\begin{proof} (1)$\Rightarrow$ (2) Let $g_1, g_2\in \mathfrak{A}$.   We consider the monomial $U_{g_1}\cdots U_{g_1}U_{g_2}$, where $U_{g_1}$   repeats $(k-1)$ times. This monomial belongs to $\mathcal{L}(\mathfrak{A})^k$. By assumption (1), it also belongs to $\mathcal{L}(\mathfrak{A};k)$. Therefore,  there exist $f_1\leq f_2\leq \cdots \leq f_k\in \mathfrak{A}$ such that  $U_{g_1}\cdots U_{g_1}U_{g_2}=U_{f_1}U_{f_2}\cdots U_{f_k}$. From this equality,  it follows that $f_k=g_1\vee g_2$. This implies that $\mathfrak{A}$ is closed under taking  supremums, as required.

(2)$\Rightarrow$ (3) This is because every finite lattice has a unique maximal element.

(3)$\Rightarrow$ (4) Let $f$ be the unique maximal element of $\mathfrak{A}$ and put $B_i=\{a\in A_i\mid a\leq f(p_i)\}$ for $i=1,\ldots,m$. Then it follows immediately that $\mathfrak{A}=\mathrm{Hom}(P,B_1,\ldots, B_m)$.

(4)$\Rightarrow$ (5)  Let $\alpha\in \mathcal{L}(\mathfrak{A})^k$. Then there are $g_i\in \mathfrak{A}$ for $i=1,\ldots,k$ such that $\alpha=U_{g_1}\cdots U_{g_k}$. For $i=1,\ldots,m$, we let $K_i$ denote the multi-set $$K_i=\{g_1(p_i),\ldots,g_k(p_i)\}.$$ Then $K_i\subset \{A_i\}_k$  and $$\alpha=X_{p_1,K_1}\cdots X_{p_m,K_m}.$$
Given $\ell\in [k]$, we define $f_{\ell}: P\rightarrow [n]$ by $f_{\ell}(p_i)=N_{\ell}(K_i)$ for $i=1,\ldots,m$. Then  $f_1\leq f_2\leq \cdots \leq f_k$ and one easily checks that $f_i\in \mathrm{Hom}(P,B_1,\ldots,B_m)$ for $i=1,\ldots,k$. Therefore $\alpha=U_{f_1}\cdots U_{f_k}\in \mathcal{L}(\mathfrak{A};k)$.

(5)$\Rightarrow$ (1)  Automatically.
\end{proof}

Combining Proposition~\ref{equality} with Theorem~\ref{power}, we obtain the main result of this section.

\begin{Corollary} \label{powermain} For all $k\geq 1$,  $R/\mathcal{L}(P,\mathcal{A})^k$ is the quotient of $T/\mathcal{L}(P^k,\mathcal{A}^k)$ by a regular sequence of variable differences.
\end{Corollary}

We remark that if $|A_i|\leq 2$ for all $i=1,\ldots,m$, it is not difficult to see that  $\mathcal{L}(P^k,\mathcal{A}^k)$ is, in fact, the polarization of $\mathcal{L}(P,\mathcal{A})^k$. However, in general, this is not the case, as shown by the following example.
\begin{Example}\label{Ex5.4} \em Assume that  $P=\{p_1\}$ and $\mathcal{A}(1)=A_1=\{1,2,3\}$. We may write $\mathcal{L}(P,\mathcal{A})=(X_1,X_2,X_3)$, by identifying $X_{p_1,a}$ with $X_{a}$ for $a=1,2,3$. Accordingly, the minimal generators of $\mathcal{L}(P^2,\mathcal{A}^2)$ can  be listed as: $$X_1^{(1)}X_1^{(2)},\ X_2^{(1)}X_2^{(2)}, \ X_3^{(1)}X_3^{(2)},\ X_1^{(1)}X_2^{(2)},\ X_1^{(1)}X_3^{(2)},\ X_2^{(1)}X_3^{(2)}.$$
Then, $X_1^{(1)}-X_1^{(2)}$ is regular on $\mathcal{L}(P^2,\mathcal{A}^2)$, and the latter is isomorphic  to   $$\mathcal{L}'=(X_1^2, \ X_2^{(1)}X_2^{(2)}, \ X_3^{(1)}X_3^{(2)},\ X_1X_2^{(2)},\ X_1X_3^{(2)},\ X_2^{(1)}X_3^{(2)}),$$ modulo $X_1^{(1)}-X_1^{(2)}$.
Next, one may check $X_2^{(1)}-X_2^{(2)}$ is regular on $\mathcal{L}'$, and  $\mathcal{L}'$ is isomorphic to
$$\mathcal{L}''=(X_1^2, \ X_2^2, \ X_3^{(1)}X_3^{(2)},\ X_1X_2,\ X_1X_3^{(2)},\ X_2X_3^{(2)})$$  modulo $X_2^{(1)}-X_2^{(2)}$. Finally,  one may check again $X_3^{(1)}-X_3^{(2)}$ is regular on $\mathcal{L}''$, and  $\mathcal{L}''$ isomorphic to
$$\mathcal{L}'''=(X_1^2, \ X_2^2, \ X_3^2,\ X_1X_2,\ X_1X_3,\ X_2X_3)$$  modulo $X_3^{(1)}-X_3^{(2)}$.
Note that $\mathcal{L}'''=\mathcal{L}(P, \mathcal{A})^2$, it follows that $\mathcal{L}(P^2,\mathcal{A}^2)$ is isomorphic to $\mathcal{L}(P, \mathcal{A})^2$ modulo the regular sequence $X_1^{(1)}-X_1^{(2)}, X_2^{(1)}-X_2^{(2)}, X_3^{(1)}-X_3^{(2)}.$

On the other side, the polarization of $\mathcal{L}(P, \mathcal{A})^2$ is as follows:
$$I=(X_1^{(1)}X_1^{(2)},\ X_2^{(1)}X_2^{(2)}, \ X_3^{(1)}X_3^{(2)},\ X_1^{(1)}X_2^{(1)}, X_1^{(1)}X_3^{(1)},\ X_2^{(1)}X_3^{(1)}). $$
By observing the positions of $X_2^{(1)}$ and $ X_2^{(2)}$ in the generators of $I$ and $\mathcal{L}(P^2,\mathcal{A}^2)$, we can see that the permutation of variables cannot transform $\mathcal{L}(P^2,\mathcal{A}^2)$ into $I$. Therefore, $\mathcal{L}(P^2,\mathcal{A}^2)$ is not the polarization of $\mathcal{L}(P, \mathcal{A})^2$.
\end{Example}

\begin{Corollary} \label{powersc}  Keep the notation introduced in Theorem~\ref{power}   and let $J$ be the ideal generated by the regular sequence $X_{i,a}^{(s)}-X_{i,a}^{(s+1)}, i=1,\ldots,m, s=1,\ldots,k-1$. Then the multigraded minimal free resolution of $R/\mathcal{L}(P,\mathcal{A})^k$ is isomorphic to the complex $\mathbb{F}/J\mathbb{F}$. Here, $\mathbb{F}$ is the multigraded minimal free resolution of $T/\mathcal{L}(P^k,\mathcal{A}^k)$ as  given in Theorem~\ref{main4.4}.
\end{Corollary}

We remark that $\mathbb{F}/J\mathbb{F}$ is derived from $\mathbb{F}$ by substituting $X_{i,a}^{(s)}$ with $X_{i,a}$ for $s=1,\ldots,k$ and replacing $T$ by $R$. To illustrate it, we consider the following example, in which we  continue to use the notation introduced in Example~\ref{Ex5.4}.
\begin{Example}\em Let $P,\mathcal{A}, \mathfrak{A}$ be the same as in Example~\ref{Ex5.4}, and denote by $S$  the polynomial ring $K[X_i^{(j)}\mid i=1,2,3, j=1,2]$. According to  Theorem~\ref{main4.4},  $\mathcal{L}(P^2,\mathfrak{A}^2)$ has a  linear resolution as follows:
$$0\rightarrow F_2\stackrel{d_2}{\rightarrow}F_1\stackrel{d_1}{\rightarrow} F_0\rightarrow 0.$$
Here, the bases and multidegrees of the basis elements are as follows:

$F_0=\bigoplus\limits_{i=1}^6Sw_i$ and $w_i$ for $i=1,\ldots,6$ form a basis of $F_0$. The  multidegrees of of these basis elements are given by  respectively: $$X_1^{(1)}X_1^{(2)},\ X_2^{(1)}X_2^{(2)}, \ X_3^{(1)}X_3^{(2)},\ X_1^{(1)}X_2^{(2)},\ X_1^{(1)}X_3^{(2)},\ X_2^{(1)}X_3^{(2)};$$
$F_1=\bigoplus\limits_{i=1}^8Sv_i$ and $v_i$ for $i=1,\ldots,8$ form a basis of $F_1$. The  multidegrees of of these basis elements are given by  respectively: $$X_1^{(1)}X_1^{(2)}X_2^{(2)},\ X_1^{(1)}X_2^{(2)}X_3^{(2)},\ X_1^{(1)}X_1^{(2)}X_3^{(2)},\ X_2^{(1)}X_2^{(2)}X_3^{(2)},$$ and $$ X_1^{(1)}X_2^{(1)}X_2^{(2)},\ X_1^{(1)}X_2^{(1)}X_3^{(2)},\  X_1^{(1)}X_3^{(1)}X_3^{(2)},\ X_2^{(1)}X_3^{(1)}X_3^{(2)};$$
$F_2=\bigoplus\limits_{i=1}^3Su_i$ and $u_i$ for $i=1,\ldots,3$ form a basis of $F_2$. The  multidegrees of of these basis elements are given by  respectively: $$ X_1^{(1)}X_1^{(2)}X_2^{(2)}X_3^{(2)},\  X_1^{(1)}X_2^{(1)}X_2^{(2)}X_3^{(2)}, \ X_1^{(1)}X_2^{(1)}X_3^{(1)}X_3^{(2)}.$$

The differential maps are given as follows:
$$d_1=\begin{pmatrix}X_2^{(2)}&0&X_3^{(2)}&0&0&0&0&0\\ 0&0&0&X_3^{(2)}&X_1^{(1)}&0&0&0\\ 0&0&0&0&0&0&X_1^{(1)}&X_2^{(1)}\\-X_1^{(2)}&X_3^{(2)}&0&0&-X_2^{(1)}&0&0&0\\0&-X_2^{(2)}&-X_1^{(2)}&0&0&-X_2^{(1)}&-X_3^{(1)}&0\\0&0&0&-X_2^{(2)}&0&X_1^{(1)}&0&-X_3^{(1)}\end{pmatrix}$$
and $$d_2=\begin{pmatrix}-X_3^{(2)}&-X_1^{(2)}&X_2^{(2)}&0&0&0&0&0\\0&-X_2^{(1)}&0&X_1^{(1)}&-X_3^{(2)}&X_2^{(2)}&0&0\\ 0&0&0&0&0&X_3^{(1)}&-X_2^{(1)}&X_1^{(1)}\end{pmatrix}^T.$$
Here, $-^{T}$ represent the transpose.

 By substituting  $X_i^{(1)}$ and $X_i^{(2)}$ with $X_i$ for $i=1,2,3$ and replacing $S$ by $R:=K[X_1, X_2, X_3]$ in the resolution of $\mathcal{L}(P^2,\mathfrak{A}^2)$, as given above ,   we can obtain  the resolution of $\mathcal{L}(P,\mathcal{A})^2=(X_1,X_2,X_3)^2$ using Corollary~\ref{powersc}. Thus, $(X_1,X_2,X_3)^2$ has the following resolution: $$0\rightarrow G_2\stackrel{f_2}{\rightarrow}G_1\stackrel{f_1}{\rightarrow} G_0\rightarrow 0.$$
Here, the bases and multidegrees of the basis elements are as follows:

$G_0=\bigoplus\limits_{i=1}^6Rw'_i$ and $w'_i$ for $i=1,\ldots,6$ form a basis of $F_0$. The  multidegrees of of these basis elements are given by  respectively: $$X_1^2,\ X_2^2, \ X_3^2,\ X_1X_2,\ X_1X_3,\ X_2X_3;$$
$G_1=\bigoplus\limits_{i=1}^8Rv'_i$ and $v'_i$ for $i=1,\ldots,8$ form a basis of $F_1$. The  multidegrees of of these basis elements are given by  respectively: $$X_1^2X_2,\ X_1X_2X_3,\ X_1^2X_3,\ X_2^2X_3,$$ and $$ X_1X_2^2,\ X_1X_2X_3,\  X_1X_3^2,\ X_2X_3^2.$$ Note that $v_2'$ and $v_6'$ share the same multidegree.

$G_2=\bigoplus\limits_{i=1}^3Ru'_i$ and $u'_i$ for $i=1,\ldots,3$ form a basis of $F_1$. The  multidegrees of of these basis elements are given by  respectively: $$ X_1^2X_2X_3,\  X_1X_2^2X_3, \ X_1X_2X_3^2.$$

The differential maps are given as follows:
$$f_1=\begin{pmatrix}X_2&0&X_3&0&0&0&0&0\\ 0&0&0&X_3&X_1&0&0&0\\ 0&0&0&0&0&0&X_1&X_2\\-X_1&X_3&0&0&-X_2&0&0&0\\0&-X_2&-X_1&0&0&-X_2&-X_3&0\\0&0&0&-X_2&0&X_1&0&-X_3\end{pmatrix}$$
and $$f_2=\begin{pmatrix}-X_3&-X_1&X_2&0&0&0&0&0\\0&-X_2&0&X_1&-X_3&X_2&0&0\\ 0&0&0&0&0&X_3&-X_2&X_1\end{pmatrix}^T.$$

\end{Example}

\section{More simplicial Spheres}
   Let $\Delta(\mathfrak{A})$ denote  the Stanley-Reisner complex of $\mathcal{L}(\mathfrak{A})^{\vee}$.  In this section, we will establish that $\Delta(\mathfrak{A})$ is always either a simplicial ball or a simplicial sphere.  Furthermore, we will classify when it is a simplicial sphere and when it is a simplicial ball in terms of the projective dimension of $\mathcal{L}(\mathfrak{A})$.

We begin with  describing the facets of $\Delta(\mathfrak{A})$. Let $V:=\{(p_i,a)\mid i\in [m], a\in A_i\}.$  We consider $\Delta(\mathfrak{A})$ as  a simplicial complex on the vertex set $V$. Let $\mathcal{F}$ denote the set of facets of $\Delta(\mathfrak{A})$. Then, by Stanley-Reisner theory (see e.g. \cite[chapter 1.5]{HH}), we have $I_{\Delta(\mathfrak{A})}=\bigcap_{F\in \mathcal{F}} \mathfrak{p}_F$ and so $\mathcal{L}(\mathfrak{A})=I_{\Delta(\mathfrak{A})}^{\vee}=(X_{V\setminus F}\:\; F\in \mathcal{F})$. Here, $\mathfrak{p}_F$ is the ideal generated by variables $X_{p_i,a}$ with $(p_i,a)\in V\setminus F$, and $X_{V\setminus F}$ is the product of variables $X_{p_i,a}$ with $(p_i,a)\in V\setminus F$.  For each $f\in \mathfrak{A}$, we put $$\Gamma f:=\{(p,f(p))\mid p\in P\}.$$ Then, it is not difficult to see that a subset  $F$  of $V$ is a facet of  $\Delta(\mathfrak{A})$ if and only if  $F=V\setminus \Gamma f$ for some $f\in \mathfrak{A}$.

 The following result  can be proven using a similar method as in the proof of \cite[Theorem 5.1]{DFN}. For the completeness, we will include a proof below. Recall from \cite[Theorem 11.4]{Bj} that a pure shellable (constructible) complex is either a PL ball or a PL sphere if every face of codimension 1 lies in at most two facets.

\begin{Proposition} Let $P,\mathcal{A}$ and $\mathfrak{A}$ be as in Conventions~\ref{setup}. Then $\Delta(\mathfrak{A})$ is a either a simplicial ball or a simplicial sphere.
\end{Proposition}

\begin{proof} Since $\mathcal{L}(\mathfrak{A})$ has linear quotients, it follows from \cite[Proposition 8.2.5]{HH} that $\Delta(\mathfrak{A})$ is a pure shellable complex. Let $G$ be a face of $\Delta(\mathfrak{A})$ of codimension 1. Then $|V\setminus G|=m+1$, and  there is exactly one element   $k$ in $[m]$  such that $|A_k\cap (V\setminus G)|=2$ and $|A_j\cap (V\setminus G)|=1$ for all $j\neq k$ with $j\in [m]$. Therefore, if $V\setminus \Gamma f$ is a facet of $\Delta(\mathfrak{A})$  containing $G$,   then $\Gamma f \subseteq V\setminus G$, and thus $f(p_i)\in A_i\cap (V\setminus G)$ for all $i\in [m]$.  It can be easily seen that there are at most two such $f$, and  the the result follows by  \cite[Theorem 11.4]{Bj}.
\end{proof}

\begin{Example} \em  Let $P=\{p,q\}$ be a poset with $p<q$, and  assume $A_{p}=[n]$ and $A_{q}=\{n\}$. Then $\Delta(\mathfrak{A})$ is an $(n-2)$-dimensional (simplicial) sphere. This is because the facets of $\Delta(\mathfrak{A})$ are all $(n-1)$-subsets of the $n$-set $\{(p,i)\mid i\in [n]\}$.
\end{Example}

\begin{Example}\label{ball} \em  Let $P=\{p,q\}$ be a poset with $p<q$, and  assume $A_{p}=[2]$ and $A_{q}=\{1,3\}$. Then, the facets of $\Delta(\mathfrak{A})$ are $F_1=\{(p,2),(q,3)\}$, $F_2=\{(p,2), (q,1)\}$ and $F_3=\{(p,1), (q,1)\}$. It is easy to verify that $\Delta(\mathfrak{A})$ is a simplicial  ball, and its boundary consists  of two points $(p,1)$ and $(q,3)$.
\end{Example}

Let $\Delta$ and $\Gamma$ be simplicial complexes on disjoint vertex sets. Then the {\it join} of $\Delta$ and $\Gamma$, denoted by $\Delta*\Gamma$, is the simplicial complex in which every face is the union of a face from $\Delta$ and a face from $\Gamma$. It is known from \cite[Equation 9.12]{Bj} that \begin{equation} \label{equ}\widetilde{H}_i(\Delta*\Gamma)\cong \bigoplus_{j+k=i-1}\widetilde{H}_j(\Delta)\otimes_{\mathbf{K}} \widetilde{H}_k(\Gamma).\end{equation}
\begin{Lemma} \label{join} Let $\Delta$ and $\Gamma$ be simplicial complexes on disjoint vertex sets.
Then \begin{enumerate}
       \item $\Delta*\Gamma$ is a homology sphere if and only if both $\Delta$ and $\Gamma$ are homology spheres.
       \item if $\Delta$ is a homology ball and $\Gamma$ is a homology sphere then $\Delta*\Gamma$ is a homology ball as well with $\partial (\Delta*\Gamma)=(\partial \Delta)*\Gamma$.
       \item if both $\Delta$ and $\Gamma$ are homology balls then $\Delta*\Gamma$ is a homology ball  with boundary of $\partial (\Delta*\Gamma)=(\partial \Delta)*(\partial \Gamma)$.
     \end{enumerate}
 \end{Lemma}
\begin{proof} (1) By  definition, every link of a homology sphere is also a homology sphere.  Note that both $\Delta$ and $\Gamma$ are links of $\Delta*\Gamma$. So if $\Delta*\Gamma$ is a homology sphere then both $\Delta$ and $\Gamma$ are homology spheres. Conversely, letting $F=G\cup H$ with $G\in \Delta$ and $H\in \Gamma$, then $\mbox{Link}_{\Delta*\Gamma}F=\mbox{Link}_{\Delta}G*\mbox{Link}_{\Gamma}H$. Note that $\mathrm{dim}(\Delta*\Gamma)=\mathrm{dim}(\Delta)+\mathrm{dim}(\Gamma)+1$, and it follows from Formula (\ref{equ}) that  $$\widetilde{H}_i(\mbox{Link}_{\Delta*\Gamma}F)=\left\{
                                                                                                                                                                                                                               \begin{array}{ll}
                                                                                                                                                                                                                                 \mathbf{K}, & \hbox{$i=\mathrm{dim} (\mathrm{Link}_{\Delta*\Gamma}(F))$;} \\
                                                                                                                                                                                                                                 0, & \hbox{otherwise.}
                                                                                                                                                                                                                           \end{array}
                                                                                                                                                                                                                             \right.
$$
This implies that $\Delta*\Gamma$ is a homology sphere.

(2) If $F\in (\partial \Delta)*\Gamma$, then we may write $F=G\cup H$ with $G\in (\partial \Delta)$ and $H\in \Gamma$. Since $\mbox{Link}_{\Delta}G$ has zero homologies, it follows that $\mbox{Link}_{\Delta*\Gamma}F=\mbox{Link}_{\Delta}G*\mbox{Link}_{\Gamma}H$ also has zero homologies; If $F$ is a face of $\Delta*\Gamma$ with $F\notin (\partial \Delta)*\Gamma$, then $F=G\cup H$ with $G\in (\Delta \setminus \partial \Delta)$ and $H\in \Gamma$, thus  $\mbox{Link}_{\Delta}G$ and $\mbox{Link}_{\Gamma}H$ are both homology spheres. It follows that $\mbox{Link}_{\Delta*\Gamma}F$ is a homology sphere by (1). Therefore $\Delta*\Gamma$ is a homology ball.

(3) This can be proved in a similar manner as in (2).
\end{proof}

If there exists $a\in A_i$ for some $i\in [m]$ such that $f(p_i)\neq a$ for all $f\in \mathfrak{A}$, then $(p_i,a)$ belongs to every facets of $\Delta(\mathfrak{A})$. This implies $$\Delta(\mathfrak{A})=\mbox{Link}_{\Delta(\mathfrak{A})}(p_i,a)*\langle\{(p_i,a)\}\rangle,$$ and so $\Delta(\mathfrak{A})$ is  a simplicial ball by Lemma~\ref{join}.  Hence, it is more reasonable to consider $\Delta(\mathfrak{A})$ as a simplcial complex on the vertex set $$V':=\{(p_i,a)\mid i=1,\ldots,m, a\in B_i\},$$ where  $B_i:=\{a\in A_i\mid f(p_i)=a \mbox{ for some } f\in \mathfrak{A}\}.$
Under this assumption, we have the following result which includes  Example~\ref{ball} as a special case.

\begin{Proposition} \label{sphere} Let $B_1,\ldots,B_m$ and $V'$ be as defined above and consider $\Delta(\mathfrak{A})$ as a simplicial complex on the vertex set $V'$. Then $\Delta(\mathfrak{A})$ is a simplicial sphere if and only if $\max{(B_i)}\leq \min(B_j)$ whenever $p_i\leq p_j$ and $\mathfrak{A}=\mathrm{Hom}(P,B_1,\ldots,B_m)$.
\end{Proposition}
\begin{proof}If $\max{(B_i)}\leq \min(B_j)$ whenever $p_i\leq p_j$ and $\mathfrak{A}=\mathrm{Hom}(P,B_1,\ldots,B_m)$, then $\Delta (\mathfrak{A})$ is the join $\partial\langle B_1\rangle*\partial\langle B_2\rangle*\cdots *\partial \langle B_m\rangle$. Here, $\langle B_i\rangle$ denotes the simplex on $B_i$. Hence, $\Delta (\mathfrak{A})$ is a simplicial sphere by Lemma~\ref{join}.

Conversely, we assume that $\Delta(\mathfrak{A})$ is a simplicial sphere. Suppose on the contrary that  $\max(B_i)>\min(B_j)$ for some $p_i,p_j\in P$ with $p_i<p_j$.
 Let $t=\max(B_{i})$, and  let $G=(V\setminus \Gamma f)\setminus \{(p_i,t)\}$. Here, $f$ is a map belonging to $\mathfrak{A}$ such that $f(p_j)=\min(B_j)$. (The existence of such  $f$ follows from the definition of $B_j$.) Then $G$ is a face of $\Delta(\mathfrak{A})$ of codimension 1. We claim that $\mathrm{Link}_{\Delta(\mathfrak{A})}G$ is a simplex of dimension 0.

 Let $g\in \mathfrak{A}$  such that $G\subseteq (V\setminus \Gamma g)$. Then we have $\Gamma g\subseteq \Gamma f\cup \{(p_i,t)\}$. This implies that $g(p)=f(p)$ if $p\neq p_i$ and $g(p_i)\in \{f(p_i),t\}$.
 Since  $t>f(p_j)=g(p_j)$ by assumption, it follows that  $g(p_i)=f(p_i)$. Therefore, we have  $f=g$, which implies $\mathrm{Link}_{\Delta(\mathfrak{A})}G$ is just a point (i.e., a simplex of dimension 0), as claimed. This leads to a contradiction, and thus we conclude $\max{(B_i)}\leq \min(B_j)$ whenever $p_i\leq p_j$.  It remains to show that $\mathfrak{A}=\mathrm{Hom}(P,B_1,\ldots,B_m)$.

Again, we suppose on the contrary that $\mathfrak{A}\neq \mathrm{Hom}(P,B_1,\ldots,B_m)$. Let $g$ be a minimal element in $\mathrm{Hom}(P,B_1,\ldots,B_m)\setminus \mathfrak{A}$. Then, since every map $f:P\rightarrow [n]$ satisfying $f(p_i)\in B_i$ for all $i\in [m]$  is isotone, there exists $f\in \mathfrak{A}$ such that $f$ and $g$ differ by a single entry. Namely, there exists $i\in [m]$ such that $f(p_i)<g(p_i)$ and $f(p_j)=g(p_j)$ for all $j\neq i$. Put $G:=(V\setminus \Gamma f)\setminus \{(p_i, g(p_i))\}$. Then the link of  $G$ is also just a point, which leads to a contradiction again. This completes the proof.\end{proof}

We conclude this paper with the following result, which classifies when $\Delta(\mathfrak{A})$ is a simplicial sphere in terms of  $\mathrm{pd}(\mathcal{L}(\mathfrak{A}))$.   Here,  $\mathrm{pd}(\mathcal{L}(\mathfrak{A}))$ denotes the  projective dimension of $\mathcal{L}(\mathfrak{A})$.

\begin{Corollary} Under the same assumptions as in Proposition~\ref{sphere}, we have:
\begin{enumerate}
  \item $\mathrm{pd}(\mathcal{L}(\mathfrak{A}))\leq |B_1|+\cdots+|B_m|-m;$
  \item $\mathrm{pd}(\mathcal{L}(\mathfrak{A}))=|B_1|+\cdots+|B_m|-m$ if and only if $\Delta(\mathfrak{A})$ is a simplicial sphere.
\end{enumerate}
\end{Corollary}
\begin{proof} It follows from Theorem~\ref{main4.4} together with Proposition~\ref{sphere}.
\end{proof}

{\bf \noindent Acknowledgment:}
This research is supported by NSFC (No. 11971338). We would like to express our heartfelt thanks to the reviewer for  careful reading, insightful comments and invaluable suggestions, which lead to much better presentation of our paper.

\end{document}